\documentclass[12pt]{amsart}
\usepackage{amsmath}
\usepackage{amsfonts}
\usepackage{amssymb}
\usepackage{amsthm}
\usepackage[all]{xy}
\usepackage{color}
\usepackage{verbatim}
\usepackage{graphicx}
\usepackage{tikz}
\usepackage{placeins}
\usepackage{float}
\usepackage{listings}

\newtheorem{theorem}{Theorem}[section]
\newtheorem{lemma}[theorem]{Lemma}

\newtheorem{corollary}[theorem]{Corollary}
\newtheorem{definition}[theorem]{Definition}

\numberwithin{equation}{section}

 \author[lokenath Kundu]{Lokenath Kundu}
 
 \email{bholaktw2010@gmail.com, lokenath$\_$kundu@srmap.edu.in}

    \address{SRM University, A.P.}
  

\title[Linear Fractional group as Galois group] {Linear Fractional group as Galois group}

\date{16/06/21}
\begin{document}
\begin{abstract}
We compute all signatures of $PSL_2(\mathbb{F}_7)$, and $PSL_2(\mathbb{F}_{11})$ which classify all orientation preserving actions of the groups $PSL_2(\mathbb{F}_7)$ and $PSL_2(\mathbb{F}_{11})$ on compact, connected, orientable surfaces with orbifold genus $\geq ~ 0$.
This classification is well-grounded in the other branches of Mathematics like topology, smooth, and conformal geometry, algebraic categories, and it is also directly related to the inverse Galois problem.

\end{abstract}
\maketitle

\section{Introduction}
\noindent Let $\Sigma_g$ be a compact, connected, orientable Riemann surface of genus $g~\geq ~ 0$. Any orientation preserving action of a finite group $G$ on a Riemann surface $\Sigma_g$ of genus $g$ gives an orbit space $\Sigma_h ~ := \Sigma_g/G$ also known as orbifold. We can take this action as conformal action, that means the action is analytic in some complex structure on $\Sigma_g$, as the positive solution of Nielson Realization problem \cite{niel,eck} implies that if any group $G$ acts topologically on $\Sigma_g$ then it can also act conformally with respect to some complex structure. \\
The orbit space $\Sigma_h$ is again a Riemann surface possibly with some marked points and the quotient map $p~:~\Sigma_g~\rightarrow~\Sigma_h$ is a branched covering map. Let $B=~\lbrace c_1,c_2,\dots,c_r~ \rbrace$ be the set of all branch points in $\Sigma_h$ and $A:=p^{-1}(B)$. Then  $p:~\Sigma_g \setminus A ~\rightarrow ~\Sigma_h \setminus B$ is a proper covering. The tuple $(h;m_1,m_2,\dots,m_r)$ is known as signature of the finite group $G$, where $m_1,m_2,\dots,m_r$ are the order of stabilizer of the preimages of the branch points  $c_1,c_2,\dots,c_r$ respectively. By Riemann-Hurwitz formula we have $$ (g-1)=~|G|(h-1)+\frac{|G|}{2}\sum_{i=1}^r(1-\frac{1}{m_i}) .$$ The signature of a group encodes the information of the group action of a Riemann surface and vice versa.

\noindent A solution of the Inverse Galois problem where $PSL_2(\mathbb{F}_7)$, and $PSL_2(\mathbb{F}_{11})$ are the Galois groups over the meromorphic function field $\mathcal{M}(\Sigma_{g ~ \geq ~ 0})$ is equivalent finding a signature of $PSL_2(\mathbb{F}_7)$ and $PSL_2(\mathbb{F}_{11})$ respectively.
Meromorphic functions on a Riemann surface $Y$  are non constant holomorphic maps from $Y~to~\widehat{\mathbb{C}}$. $\mathcal{M}(Y)$ the set of all meromorphic functions on $Y$ forms a field. It is well known that $\mathcal{M}(\widehat{\mathbb{C}})~=~\mathbb{C}(t)$. For a finite group $G$, the $\mathbb{C}(t)$ isomorphism classes of Galois extensions $L/\mathbb{C}(t)$ with Galois group isomorphic to $G$ are in $1-1$ correspondence with the equivalence class of branched coverings $p~:Y~\rightarrow~\widehat{\mathbb{C}}$ with the Deck transformation group isomorphic to $G$. In the above correspondence $L~\cong\mathcal{M}(Y)$.  (see \cite{ggal}).

\noindent Finding signatures of a group gives a positive answer to the inverse Galois problem. For any prime number $p$, $PSL_2(\mathbb{F}_p)$ is a group of $2 \times 2$ matrices over the fields  $\mathbb{F}_{p}$ with diterminant $1$, modulo its center, and is of order $\frac{(p-1)p(p+1)}{2}$. In this article we consider the groups $PSL_2(\mathbb{F}_{7})$ and $PSL_2(\mathbb{F}_{11})$, and we compute all the signatures of the mentioned groups. The order of any element in $PSL_2(\mathbb{F}_7)$ is either of 1, 2, 3, 4 or 7. Also the six conjugacy classes of $PSL_2(\mathbb{F}_7)$ are denoted by 1A, 2A, 3A, 4A, 7A and 7B size 1, 21, 56, 42, 24 and 24 respectively. The maximal subgroups of $PSL_2(\mathbb{F}_7)$ are $S_4$ and $\mathbb{Z}_7\rtimes\mathbb{Z}_3$. We have 14 copies of $S_4$ and 8 copies of $\mathbb{Z}_7\rtimes\mathbb{Z}_3$ in $PSL_2(\mathbb{F}_7)$ respectively. The order of any element in $PSL_2(\mathbb{F}_{11})$ is either of 1, 2, 3, 5, 6 or 11. Also the six conjugacy classes of $PSL_2(\mathbb{F}_{11})$ are denoted by 1A, 2A, 3A, 5A, 5B, 6A, 11A and 11B and are size 1, 55, 110, 132, 132, 110, 60 and 60 respectively. The maximal subgroups of $PSL_2(\mathbb{F}_{11})$ are $A_5$,  $\mathbb{Z}_{11}\rtimes\mathbb{Z}_5$, and $D_{12}$. We have 22 copies of $A_5$ and 12 copies of $\mathbb{Z}_{11}\rtimes\mathbb{Z}_5$, and 55 copies of $D_{12}$  in $PSL_2(\mathbb{F}_{11})$ respectively \cite{atlas}. The maximal subgroups and the number of maximal subgroups in $PSL_2(\mathbb{F}_7)$ and $PSL_2(\mathbb{F}_{11})$ will help us to compare the number of solutions of the equation $xyz=1$ in $PSL_2(\mathbb{F}_7)$ and $PSL_2(\mathbb{F}_{11})$ respectively and in maximal subgroups of $PSL_2(\mathbb{F}_{7})$ and $PSL_2(\mathbb{F}_{11})$ respectively, which allow us to compute the signatures of the form $(0;m_1,m_2, \dots,m_r)$ of the groups $PSL_2(\mathbb{F}_{7})$ and $PSL_2(\mathbb{F}_{11})$.

\noindent It was justified in \cite{isomat} that $PSL_2(\mathbb{F}_7) $ is isomorphic to $GL_3(\mathbb{F}_2)$. Besides that $GL_3(\mathbb{F}_2)$ is isomorphic to $PSL_3(\mathbb{F}_2)$. We use this isomorphism to prove our main result. So $PSL_2(\mathbb{F}_7) $ is isomorphic to $PSL_3(\mathbb{F}_2)$. In \cite{ming2} the authors computed the genus of the linear fractional groups $PSL_2(\mathbb{F}_{p^n})$ for all primes $p$ and exponents $n$.  In \cite{ming1} the authors showed that the minimum genus of $PSL_2(\mathbb{F}_7)$ and $PSL_2(\mathbb{F}_{11})$ are given by the signature $(0;2,3,7)$ and $(0;2,3,11)$ respectively. Using the Riemann-Hurwitz formula we can easily see that the minimum genus of $PSL_2(\mathbb{F}_7)$ is $3$ and $PSL_2(\mathbb{F}_{11})$ is $26$. In \cite{klein} the author shows that minimum genus of $PSL_2(\mathbb{F}_7)$ is $3$ using algebraic functions, and branched points.
Later in \cite{oza} Murad {\"O}zaydin, Charlotte Simmons, and Jennifer Taback determined all the signatures of the form $(0;m_1,m_2,\dots,m_r)$ for the groups $PSL_2(\mathbb{F}_{p})$ for $p \geq 11$. Also in this process of proving their results they claimed that $(0;2,2,2,2,2)$ is not a signature for the group $PSL_2(\mathbb{F}_7)$ and 
$PSL_2(\mathbb{F}_{11})$. In connection with the signature of a finite group and inverse Galois problem, we have the following known results.
In 1980 Samuel E. LaMacchia gives a two parameter family $f(a,~ t)(X)$ of
polynomials over $\mathbb{Q}(a,~ t)$ with the Galois group $PSL_2(\mathbb{F}_7)$. In \cite{malle} the author gives $3$- parameter family polynomial for $PSL_2(\mathbb{F}_7)$ for specific signatures. The necessary and sufficient condition for an effective action of a group $G$ preserving the orientation on compact, connected,  orientable surface $\Sigma_g$ of genus $g$ except for finitely many exceptional values of $g$ was proved by Kulkarni in \cite{kulkarni}. In particular the group $PSL_2(\mathbb{F}_p)$  has the above mentioned property for $p \geq ~ 5$, and $p$ is odd.

\noindent For convenience, we write $(h;2^{[2]},3^{[1]})$ for the signature $(h;2,2,3)$. In general we use the following notation $ ( h;2^{[a_{2}]}, 3^{[a_{3}]}, 4^{[a_{4}]}, 7^{[a_{7}]} )$ for $PSL_2(\mathbb{F}_7)$ and $ ( h;2^{[a_{2}]}, 3^{[a_{3}]}, 5^{[a_{4}]}, 6^{[a_{6}]}, 11^{[a_{11}]} )$ for $PSL_2(\mathbb{F}_{11})$. 

\noindent In \cite{oza} the authors claimed that $(0;2,2,2,2,2)$ is not a signature of  $PSL_2(\mathbb{F}_7)$ and $PSL_2(\mathbb{F}_{11})$. We will show that this is not the case,
 below we state our main results.
\begin{theorem}\label{main theorem 1}
$ ( h;2^{[a_{2}]}, 3^{[a_{3}]}, 4^{[a_{4}]}, 7^{[a_{7}]} ) $ is a signature of $ PSL_2(\mathbb{F}_7) $ if and only if $$  1+168(h-1)+ 42a_{2} + 56a_{3} + 63a_{4} + 72a_{7} \geq 3 $$ except when the signature is $(1;2)$.
 \end{theorem} 
 
\begin{theorem}\label{main theorem 2}
$ ( h;2^{[a_{2}]}, 3^{[a_{3}]}, 5^{[a_{5}]}, 6^{[a_6]} 11^{[a_{11}]} ) $ is a signature of $ PSL_2(\mathbb{F}_{11}) $ if and only if $$  1+660(h-1)+ 165a_{2} + 220a_{3} + 264a_{5} + 275a_6 +300a_{11} \geq 26 .$$
 \end{theorem}  
 
\noindent As a corollary of the main theorems we will find out the stable upper genus for the groups $PSL_2(\mathbb{F}_7)$, and $ PSL_2(\mathbb{F}_{11})$. We define the stable upper genus of any group $G$ to be the least integer $g$ such that $G$ acts on every compact, connected, orientable Riemann surface $\Sigma_{g+i}, ~ \forall ~ i=0,1,2,\dots$ effectively. Here we find the stable upper genus using generic programming techniques \cite{program,ipython,pandas,matplotlib,numpy} for an extensive range of $h ~ \text{and} ~ a_i$ where we systematically use combinations of nested loops and high dimensonal arrays for our purpose.  We will prove that only two irreducible characters of $PSL_2(\mathbb{F}_7)$ come from the Riemann surfaces and there is no irreducible characters of $PSL_2(\mathbb{F}_{11})$  come from the Riemann surfaces. Using our main theorem we also find out all possible Galois extension over $\mathcal{M}(\Sigma_{g \geq 0})$ with Galois groups $PSL_2(\mathbb{F}_7)$ and $PSL_2(\mathbb{F}_{11})$.The group $ PSL_2(\mathbb{F}_p),$ where $p$ is a prime with $p\leq7$ has an application in physics as well. An interesting application of the representations of $ PSL_2(\mathbb{F}_7)$  have been used, in \cite{phy}. Our results could profitably be applied to find similar interesting applications of the representations of $PSL_2(\mathbb{F}_7)$.

\noindent The paper is organized as follows. In Section $2,$ we discuss background results. In Section $3,$ and Section $4,$ we will prove our main theorem and will show some applications.
\section{Preliminaries}
\noindent In this section, we will collect the knowledge about the branched covering, meromorphic function fields over Riemann surfaces, Fuchsian group, signature of the finite group, class multiplication coefficient formula, and Eichler trace formula. We are going to study the relation between the branched covers and meromorphic function fields.
 
\noindent We start with the definition of branch points and the relation between branched covering and inverse Galois problem. 
\begin{definition} \cite{otto}
Consider any two Riemann surfaces $X~ \text{and} ~Y$. Let $f:~Y ~ \rightarrow ~ X$ be a non constant holomorphic mapping. A point $y~\in~Y$ is said to be a ramification point if there exists no neighbourhood $V$ of $y$ in $Y$ such that $f_{|V}$ is $1-1$ and $f(y)$ is known as branch point
\end{definition} 
\begin{definition} \cite{ggal}
Let $P$ be a finite subset of $\mathbb{CP}^1$. Let $f:R ~ \rightarrow ~ \mathbb{CP}^1 \setminus  ~P $ be a finite Galois covering. Let $H ~ = ~ Deck(f)$, and for $p \in P$ let $C_p$ be the associated conjugacy classes of $H$. Let $P^{\prime} ~ = ~ \lbrace p \in P|~ C_p \neq \lbrace 1 \rbrace ~ \rbrace$. We define the ramification type of $f$ to be the class of triple $(H,P^{\prime},(C_p)_{p\in P^{\prime}})$.
\end{definition}
\begin{theorem} \cite{ggal} {(Riemann existence theorem- topological version)}
Let $\tau ~ = [G,~ P, ~ (K_p)_{p \in P}]$ be a ramification type. Let $|P|=r$, we label the elements of $P$ as $p_1,p_2,\dots,p_r$. Then there exists a finite Galois covering of $\mathbb{CP}^1 \setminus  ~P$ of ramification type $\tau$ if and only if there exists generators $g_1,g_2.\dots,g_r$ of $G$ with $g_1 ~ \dots ~g_r=1$ and $g_i ~ \in ~ K_{p_i}$ for $i=1,2,\dots, r$. 
\end{theorem}
\begin{theorem} \cite{ggal}
Let $G$ be a finite group. Let $P$ be a finite subset of $\mathbb{CP}^1$, and $q \in \mathbb{CP}^1 \setminus ~P$. There is a natural $1-1$ correspondence between the following objects.
\begin{itemize}
\item[1.] The $\mathbb{C}(x)-$ isomorphism class of Galois extension $L/ \mathbb{C}(x)$ with Galois group isomorphic to $G$ and with branch points contained in $P$.
\item[2.] The equivalence classes of Galois covering $f:R ~ \rightarrow ~ \mathbb{CP}^1 \setminus  ~P $ with Deck transformation group isomorphic to $G$.
\item[3.] The normal subgroups of the fundamental group $\pi_1(\mathbb{CP}^1 \setminus  ~P,q)$ with quotient isomorphic to $G$. 
\end{itemize}
\end{theorem}
\noindent Now we are going to study the alliance of the branch points with character theory of finite group.

\noindent A discrete subgroup of $PSL_2(\mathbb{R})$ is known as \emph{Fuchsian group} \cite{fuch,sve}. The following theorem narrates a presentation for Fuchsian group.
\begin{theorem}
{\cite{tb}} If $ \Gamma $ is a Fuchsian group with compact orbit space $ \mathcal{U}/ \Gamma $ of genus $h$ then there are elements $ \alpha_{1},\beta_{1}, \dots ,\alpha_{h},\beta_{h},c_{1},\dots,c_{r} $ in $ Aut(\mathcal{U}) $ such that the following holds,

\begin{enumerate}
\item We have $ \Gamma = \langle \alpha_{1},\beta_{1}, \dots ,\alpha_{h},\beta_{h},c_{1},\dots,c_{r} \rangle. $
\item Defining relations for $ \Gamma $ are given by $$ c_{1}^{m_{1}},\dots,c_{r}^{m_{r}},\prod_{i=1}^{h}[\alpha_{i},\beta_{i}] \prod_{j=1}^{r}c_{j} $$ where the $ m_{i} $ are integers with $ 2\leq m_{1}\leq \dots \leq m_{r}. $
\item  Each non identity element of finite order in $ \Gamma $ lies in a unique conjugate of $ \langle c_{i} \rangle $ for suitable $i$ and the cyclic subgroups $\langle c_i\rangle$ are maximal in $ \Gamma. $ 
\item Each non identity element of finite order in $ \Gamma $ has a unique fixed point in $ \mathcal{U}. $ Each element of infinite order in $ \Gamma $ acts fixed point freely on $ \mathcal{U} .$ Finite order elements are called elliptic elements and infinite order element are called hyperbolic elements.  
\end{enumerate}
\end{theorem}

\noindent For a Fuchsian group $ \Gamma $ as above, the numbers $ h,r $ and $ m_{1},m_{2},\dots,m_{r} $ are uniquely determined; $ 2h $ is the rank of free abelian part of the commutator factor of $ \Gamma. $ \\
We call $\sigma:= (h;m_{1},m_{2},\dots,m_{r}) $ the signature of $ \Gamma, $ the integer $ h $ is called the orbit genus and $ m_{i} $ are called the periods of $ \Gamma $ for $i=1,2,\dots,r. $ We define $$\Gamma(\sigma)=\langle \alpha_{1},\beta_{1}, \dots ,\alpha_{h},\beta_{h},c_{1},\dots,c_{r}\\ |
c_{1}^{m_{1}}=\dots,=c_{r}^{m_{r}}=\prod_{i=1}^{h}[\alpha_{i},\beta_{i}] \prod_{j=1}^{r}c_{j}=1\rangle $$

\noindent Now we define the signature of a finite group in the sense of Harvey \cite{har}.

\begin{lemma}[Harvey condition]
\label{Harvey condition}
A finite group $G$ acts faithfully on $\Sigma_g$ with signature $\sigma:=(h;m_1,\dots,m_r)$ if and only if it satisfies the following two conditions: 

\begin{enumerate}

\item The \emph{Riemann-Hurwitz formula for orbit space} i.e. $$\displaystyle \frac{2g-2}{|G|}=2h-2+\sum_{i=1}^{r}\left(1-\frac{1}{m_i}\right), \text{ and }$$

\item  There exists a surjective homomorphism $\phi_G:\Gamma(\sigma) \to G$ that preserves the orders of all torsion elements of $\Gamma$. The map $\phi_G$ is also known as surface-kernel epimorphism.
\end{enumerate}
\end{lemma}

\noindent Suppose we have a signature $(h;m_{1},m_{2},\dots,m_{r})$ for a given group $G$. Extension principle gives us a clear idea when $(h;m_{1},m_{2},\dots,m_i^{\prime},m_i^{\prime \prime},\dots,m_{r})$ is again a signature of the group G.
\subsection{Extension principle}
\begin{theorem}
Let $(h;m_{1},m_{2},\dots,m_{r})$ be a signature of a finite group $G=~PSL_2(\mathbb{F}_p)$ where $p$ is a prime number. Now $$\Gamma(h;m_{1},m_{2},\dots,m_{r})=\langle \alpha_{1},\beta_{1}, \dots ,\alpha_{h},\beta_{h},c_{1},\dots,c_{r} | c_{1}^{m_{1}},\dots,c_{r}^{m_{r}},\prod_{i=1}^{g}[\alpha_{i},\beta_{i}] \prod_{j=1}^{r}c_{j}\rangle. $$
If $c_i=c_{i1}c_{i2} \in ~ \Gamma; i \in \lbrace1,2,\dots,r\rbrace$ where $|c_{i1}|=m_i$ and $|c_{i2}|=m_i$ where $m_i$ is odd prime or $m_i=2$, then $$(h;m_{1},m_{2},\dots,m_{i},m_{i},\dots,m_{r})$$ is also a signature of $G$. It is known as \emph{extension principle}.
\end{theorem}

\begin{proof}
Let $(h;m_{1},m_{2},\dots,m_{r})$ be a signature of  finite group $G$. Then there exists a surface kernel epimorphism $$\phi: ~ \Gamma(h;m_{1},m_{2},\dots,m_{r}) ~ \rightarrow ~ G ~ \text{defined as } $$ $$ \phi (\alpha_i)=A_i$$ $$\phi (\beta_i)=B_i$$ $$\phi(c_i)=C_i.$$
As $c_i=c_{i1}c_{i2}$ where $|c_{i1}|=m_{i}$ and $|c_{i2}|=m_{i}$. Now $$\phi(c_i)=C_i ~ \Rightarrow ~ \phi(c_{i1}c_{i2})=C_i ~ \Rightarrow ~ \phi(c_{i1})\phi(c_{i1})=C_i$$ $$\Rightarrow C_{i1}C_{i2}= C_i ~ \text{where}~ \phi(c_{i1})=C_{i1} ~ \text{and}~ \phi(c_{i2})=C_{i2}.$$ As $\phi$ is surface kernel epimorphism, so $|C_{i1}|=m_{i}$ and $|C_{i2}|=m_{i}$. Now we can define the following map $$\psi :~ \Gamma(h;m_{1},m_{2},\dots,m_{i1},m_{i2},\dots,m_{r}) ~ \rightarrow G ~ \text{as}$$ $$ \psi(\alpha_j)=A_j, ~ \psi(\beta_j)=B_j,~ \psi(c_j)=C_j ~ \text{for} ~ j \neq i $$  and $$ \psi(c_{i1})=C_{i1}, ~ \text{and}~\psi(c_{i2})=C_{i2}.$$ Clearly $\psi $ is surface kernel epimorphism. \\
\noindent Now we have to show that 
$$1+|G|(h-1)+\frac{|G|}{2}\sum_{i=1}^r (1-\frac{1}{m_i})+\frac{|G|}{2}(1-\frac{1}{m_i})$$ is an integer.
As $(h;m_1,m_2,\dots,m_r)$ is a signature of $|G|$. So $$1+|G|(h-1)+\frac{|G|}{2}\sum_{i=1}^r (1-\frac{1}{m_i})$$ is an integer. Now $\frac{|G|}{2}(1-\frac{1}{m_i})$ is an integer if $m_i=2$ or $m_i$ is an odd prime dividing $\frac{(p-1)p(p+1)}{2}$. Hence $(h;m_{1},m_{2},\dots,m_{i},m_{i},\dots,m_{r})$ is a signature of $G$.
\end{proof}
\subsection{Minimum Genus}
Let $G$ be a finite group. The genus spectrum of a finite group $G$  is denoted by $SP(G)$, and it is defined as $$SP(G)=\lbrace ~ g \geq ~ 0 |~ G \text{ acts on }\Sigma_g \text{ faithfully, preserving orientation }\rbrace.$$ 
\noindent We define the minimum genus $\mu(G)$ of group $G$ by
\begin{center}
$\mu(G)=\operatorname{min} SP(G)$.
\end{center}
In our case the minimum genus of $PSL_2(\mathbb{F}_7)$ is $3$ \cite{ming2}, \cite{ming1}, \cite{klein}  and $PSL_2(\mathbb{F}_{11})$ is 26 \cite{ming2}, \cite{ming1}.
 \subsection{Class multiplication coefficient formula}
 Let G be a group, $x,y,z \in G.$ The following theorem provide the number of solutions of the equation $xyz=1$ in $G.$ The following theorem will help us to compute the signatures of the form $(0;m_1,m_2,\dots,m_r)$ of a group $G$.
 \begin{theorem}\cite{fing, char}
 We denote the conjugacy classes of $G$ by $K_i$ and let $y_i$ be an element of $K_i,~1\leq i\leq r.$ Then $\lambda_{ijk}$ is the number of times a given element of $K_k$ can be expressed as an ordered product of an element of $K_i$ and an element of $K_j,$ we have
 $$\lambda_{ijk}=\frac{|K_i| \cdot|K_j|}{|G|}\sum_{\chi \in Irr(G)}\frac{\chi(y_i)\chi(y_j)\overline{\chi(y_k)}}{\chi(1)} $$
 Here $Irr(G)$ denotes the set of all irreducible characters of a group.
 \end{theorem}
\noindent We can easily derive the following corollary from the above theorem.
\begin{corollary}
 Let $G$ be a finite group, and let $C_1,C_2,\dots,C_k$ denote the conjugacy classes in $G$, with representatives $g_1,g_2,\dots,g_k$. Then the number of solutions to the equation $xyz=1$ in $G$ with $x\in C_r,y\in C_s \text{ and } z\in C_t$ is given by 
 \begin{equation}
     C_{C_r,C_s,C_t}=\frac{|C_r| \cdot |C_s| \cdot |C_t|}{|G|}\cdot \sum_{\chi \in Irr(G)} \frac{\chi(g_r)\chi(g_s)\chi(g_t)}{\chi(1)},
 \end{equation}
  Here $Irr(G)$ denotes the set of all irreducible characters of a group. The $ C_{C_r,C_s,C_t} $ is called as the associated class multiplication coefficient.
 \end{corollary}
 \par
\noindent We are interested to find the irreducible characters of a finite group $G$ coming from the Riemann surfaces. \emph{Eichler trace formula} is a recipe to cook up those characters. We need the following results to state \emph{Eichler trace formula.}   

\begin{lemma} \cite{tb} Let X be a compact Riemann Surface of genus at least $ 2 $ and $ \sigma \in Aut(X) $ of order $ m \geq 2 $ such that $  \sigma(x):=x^{\sigma}=x $ for a point $ x \in X. $ 
 Then there exists unique primitive $m-th$ root of unity $ \zeta $ such that any lift $ \tilde{\sigma} $ of $ \sigma $ to $ \mathcal{D},\text{ (the open unit disc) } $ that fixes a point in $ \mathcal{D} $ is conjugate to the transformation $ z \longmapsto \zeta z. $ \\
 We write $\zeta_{x} (\sigma)=\zeta.$ $ \zeta^{-1} $ the \emph{rotation constant} of $ \sigma $ in $X$.  
\end{lemma} 
    
 \begin{definition}
 Let X be a compact Riemann surface, $ h \in G \leq Aut(X) $ of order m, and $ u \in \mathbb{Z} $ with $ gcd(u,m)=1, $ we define 
$$ Fix_{X,u}(h) :=\lbrace x \in Fix_{X}(h) \mid \zeta_{x}(h)= \zeta_{m}^{u} \rbrace, $$ the set of fixed point of $h$ with rotation constant $ \zeta_{m}^{-u}. $
\end{definition}
\noindent We denote the set of all positive integers smaller than $m$ and coprime to $m$ by $ I(m)$ i.e $I(m) = \lbrace u \mid 1 \leq u \leq m,~ gcd(u,m)=1 \rbrace.$ 
The following lemma gives us a formula to compute the order of $Fix_{X,u}(h).$ 
\begin{lemma} \cite{tb}
Let $ h \in G^{\ast} $ be of order m and $ u \in I(m)$ and $C_{G}(h)$ denote the conjugacy class of $h \in G$. Then  $$ \displaystyle{|Fix_{X,u}(h)| = |C_G(h)|  \sum_{\substack{1 \leq i \leq r \\ m \mid m_{i} \\ h \sim_{G} \phi(c_i)^{m_i u/m}}}} \frac{1}{m_i} $$
\end{lemma}
\noindent A very crucial property of characters of finite groups that come from the Riemann surface is stated as follows.
\begin{theorem}[Eichler Trace Formula] \cite{tb}
 Let $ \sigma $ be an automorphism of order $ m \geq 2 $ of a compact Riemann Surface X of genus $ g \geq 2, $ and $ \chi $ the character of the action of $ Aut(X) $ on $ \mathcal{H}^{1}(X). $ Then 
$$ \chi(\sigma) = 1+ \sum_{u \in I(m)} |Fix_{X,u}(\sigma)| \frac{\zeta_{m}^{u}}{1-\zeta_{m}^{u}}. $$ Here $ \mathcal{H}^{1}(X) $ is the set of holomorphic differentials on Riemann surface X.  $ \mathcal{H}^{1}(X) $ is a complex vector space of  dimension $g(X)$ and  $\zeta_{m}$ is the primitive m-th root of unity.
\end{theorem}
\section{Signatures of $PSL_2(\mathbb{F}_7)$}
\noindent In this section, we will prove the necessary and sufficient condition that classify all possible signature of $ PSL_2(\mathbb{F}_7)$.
\begin{lemma}
\label{lemma1}
$ (0; 2^{[a_{2}]}, 3^{[a_{3}]}, 4^{[a_{4}]}, 7^{[a_{7}]} ) $ is a signature of $ PSL_2(\mathbb{F}_7)$ if and only if
$$  1-168+ 42a_{2} + 56a_{3} + 63a_{4} + 72a_{7} \geq 3 $$
  \end{lemma}
\begin{proof}
It is an easy observation that any element of order $2,~ 3,~ 7 \text{ in }~ PSL_2(\mathbb{F}_7)$ can be written as a product of two element order $2,~ 3,~ 7$ in $PSL_2(\mathbb{F}_7)$ respectively. We will use this fact and extension principle multiple time to produce signatures of $PSL_2(\mathbb{F}_7)$ in the proof. We will carry the notations from \cite{atlas} to denote conjugacy classes, charater values.
 We only have to discuss the following three cases to prove the lemma,
\begin{enumerate}
\item If $ a_{2} + a_{3} + a_{4} + a_{7} =3. $ In this case if $a_{2} \geq 1$ then $ a_{7} \geq 1, $ else $ a_{4} + a_{7} \geq 1$. 
 \item $  a_{2} + a_{3} + a_{4} + a_{7} > 4 $ or $  a_{2} + a_{3} + a_{4} + a_{7} = 4 $ and $ a_{2} \neq 4$.
\item $  a_{2} \geq 5 $ if $a_3,a_4,a_7=0$. 

 \end{enumerate}
The minimum genus $\mu (G) = 3$ of $PSL_2(\mathbb{F}_7)$. So, $ (0; 2^{[a_{2}]}, 3^{[a_{3}]}, 4^{[a_{4}]}, 7^{[a_{7}]} ) $ is a possible signature of $ PSL_2(\mathbb{F}_7), $ then $$ 1-168+ 42a_{2} + 56a_{3} + 63a_{4} + 72a_{7} \geq 3.$$

Now we assume that $a_{2}+a_{3}+a_{4}+a_{7} =3$ and also assume $a_{7} \geq 1 \ whenever \ a_{2} \geq 1 \ and \ a_{4}+a_{7}=1.$

Now we discuss all the signatures  of the form $(0;m_1,m_2,m_3)$ one by one.

\begin{enumerate}
\item[Case 1.] Consider $(0;2,4,7)$. Now $C^{PSL_2(\mathbb{F}_7)}_{2,4,7}= ~ 168$. So if $(0;2,4,7)$ does not generate the group $PSL_2(\mathbb{F}_7)$ then the the only maximal subgroup that contain the element of order $2 ~ \text{and} ~4$ is $S_4$. but $7 \nmid ~ S_4$. this will lead to contradiction. Hence $(0;2,4,7)$ will be a signature of $PSL_2(\mathbb{F}_7)$.
\item[Case 2.] Now we consider the tuple $(0;2,7,7)$. Using coefficient multiplication formula we have $C_{2,7A,7A}^{PSL_2(\mathbb{F}_7)} =120 = C_{2,7B,7B}^{PSL_2(\mathbb{F}_7)} $ but $ C_{2,7A,7B}^{PSL_2(\mathbb{F}_7)} = 0$. If $ (0;2,7,7)$ is not a signature of $PSL_2(\mathbb{F}_7)$, then the only subgroup that contains element of order $7$ is $\mathbb{Z}_7 \rtimes \mathbb{Z}_3$ but $2 \nmid 21$. This will lead us a contradiction. So $ (0;2,7,7)$ will be a signature $ PSL_2(\mathbb{F}_7)$.
\item[Case 3.] Now consider $(0;3,3,4)$. If this tuple is not a signature of $PSL_2(\mathbb{F}_7)$ then the only maximal subgroup that contains element of order $3 ~ and ~ 4$ is $S_4$. Now from coefficient multiplication formula we have $C_{3,3,4}^{PSL_2(\mathbb{F}_7)}=~ 672 ~ and ~ 14\cdot C^{S_4}_{3,3,4}=~0$. So $(0;3,3,4)$ is a signature. Now $C_{3,4,4}^{PSL_2(\mathbb{F}_7)}=~ 672$  and $ 14. C^{S_4}_{3,4,4}=~336$. Hence $(0;3,4,4)$ is also signature by repeating the above argument.
\item[Case 4.] Now consider $(0;3,3,7)$. From coefficient multiplication formula we get $C_{3,3,7}^{PSL_2(\mathbb{F}_7)}=504$. If $(0;3,3,7)$ is not a signature of $PSL_2(\mathbb{F}_7)$ then the only maximal subgroup of $PSL_2(\mathbb{F}_7)$ is $\mathbb{Z}_7\rtimes \mathbb{Z}_3$. But $8 \cdot C_{4^{[3]}}^{\mathbb{Z}_7 \rtimes \mathbb{Z}_3}=0$. Hence $(0;3,3,7)$ is a signature of $PSL_2(\mathbb{F}_7)$. By similar kind of argument we can easily prove that $(0;3,7,7)$ is a signature of $PSL_2(\mathbb{F}_7)$. In this case $C_{3,7,7}^{PSL_2(\mathbb{F}_7)}=~ 216$ and $C_{3,7,7}^{\mathbb{Z}_7\rtimes \mathbb{Z}_3}=0$.
\item[Case 5.] Now consider the tuple $(0;3,4,7)$ and the corresponding coefficient multiplication formula $C^{PSL_2(\mathbb{F}_7)}_{3,4,7} ~ = ~ 336$. Now if the tuple does not generate $PSL_2(\mathbb{F}_7)$ then the only maximal subgroup containg elements of order $3 ~ \text{and} ~ 4 ~ \text{is} ~ S_4$. But $7 ~ \nmid 24$. Hence $(0;3,4,7)$ is a signature of $PSL_2(\mathbb{F}_7)$.
\item[Case 6.] Now consider the tuple $(0;4^{[3]})$. The only subgroup that contain element of order $4$ in $PSL_2(\mathbb{F}_7)$ is $S_4$. From the coefficient multiplication formula we have $C_{4^{[3]}}^{PSL_2(\mathbb{F}_7)}=~ 672$ but $C_{4^{[3]}}^{S_4}=~ 0$. So $(0;4,4,4)$ will be a signature of $PSL_2(\mathbb{F}_7)$.
\item[Case 7.] First consider $(0;4,4,7)$. From coefficient multiplication formula we have $C_{4,4,7}^{PSL_2(\mathbb{F}_7)}=~ 168$. As the only maximal subgroup of $PSL_2(\mathbb{F}_7)$ is $S_4$ that contain elements of order $4$. but $7 ~\nmid ~ 24$. Hence $(0;4,4,7)$ will be a signature of $PSL_2(\mathbb{F}_7)$.

\noindent Now consider $(0;4,7,7)$. Using coefficient multiplication formula we get $C_{4,7,7}^{PSL_2(\mathbb{F}_7)}=~ 336 ~ \text{or} ~ 0$. Now if possible let $(0;4,7,7)$ be not signature. Now the maximal subgroup of $PSL_2(\mathbb{F}_7)$ that contains element of order 4 is $S_4$ but $7 ~ \nmid ~ 24$. Hence $(0;4,4,7)$ is a signature of $PSL_2(\mathbb{F}_7)$.

\item[Case 8.] Now consider the tuple $(0;7^{[3]})$. The only subgroup that contain element of order $7$ in $PSL_2(\mathbb{F}_7)$ is $\mathbb{Z}_7\rtimes \mathbb{Z}_3$, and we have $8$ isomorphic copies of $\mathbb{Z}_7\rtimes \mathbb{Z}_3$ in $PSL_2(\mathbb{F}_7)$. Now $C_{7A,7A,7A}^{PSL_2(\mathbb{F}_7)}= 216 = C_{7B,7B,7B}^{PSL_2(\mathbb{F}_7)}$, and $C_{7A,7A,7B}^{PSL_2(\mathbb{F}_7)}= 24 = C_{7B,7B,7A}^{PSL_2(\mathbb{F}_7)}$ But $C_{7A,7A,7B}^{\mathbb{Z}_7\rtimes \mathbb{Z}_3}=~ 3 =C_{7B,7B,7A}^{\mathbb{Z}_7\rtimes \mathbb{Z}_3} ~\text {and} ~ C_{7A,7A,7A}^{\mathbb{Z}_7\rtimes \mathbb{Z}_3}=6=C_{7B,7B,7B}^{\mathbb{Z}_7\rtimes \mathbb{Z}_3}$. So $(0;7,7,7)$ is a signature of $PSL_2(\mathbb{F}_7)$, as $8.C_{7B,7B,7B}^{\mathbb{Z}_7\rtimes \mathbb{Z}_3} < C_{7B,7B,7B}^{PSL_2(\mathbb{F}_7)},~8. C_{7A,7A,7A}^{\mathbb{Z}_7\rtimes \mathbb{Z}_3} < C_{7A,7A,7A}^{PSL_2(\mathbb{F}_7)}$. 
\end{enumerate}

\noindent Now we consider $  a_{2} + a_{3} + a_{4} + a_{7} > 4 $ or $  a_{2} + a_{3} + a_{4} + a_{7} = 4 $ and $ a_{2} \neq 4. $
\begin{itemize}
\item[Case 1] $a_4\geq1$
\begin{itemize}
\item[Subcase 1] $a_2,a_3,a_7=0$
So $ a_{4} \geq 4. $ We already proved that $(0;4^{[3]})$ is a signature of $PSL_2(\mathbb{F}_7)$.
So there exists a suface kernel epimorphism $ \phi: \Gamma( 0;4,4,4 ) \rightarrow PSL_2(\mathbb{F}_7) $ where $ \Gamma( 0;4,4,4 ) = \langle  c_{1}, c_{2}, c_{3} \mid c_{1}^{4} = c_{2}^{4} = c_{3}^{4} = 1, c_{1}c_{2}c_{3} = 1 \rangle $ and $ \phi (c_{i}) =e_{i}  $ for $ i=1,2,3. $
Now, we define a map $ \phi: \Gamma( 0;4^{[a_{4}]} ) \rightarrow PSL_2(\mathbb{F}_7) $ where $ \Gamma( 0;4^{[a_{4}]} ) = \langle c_{1},\dots , c_{a_{4}} \mid c_{1}^{4} = \dots =c_{a_{4}}^{4}=1 , c_{1} \dots c_{a_{4}} =1  \rangle $ as follows: \\ when \emph{ $ a_{4} $ is even }
$$ \phi (c_{1} ) = e_{1},~ \phi (c_{2} ) = e_{1}^{-1}, $$ 
$$ \phi (c_{3} ) = e_{2},~ \phi (c_{4} ) = e_{2}^{-1}, $$
$$ \phi (c_{i} ) = e_{1} ;\text{i is odd and} \hspace{0.2cm} 5\leq i \leq a_{4},  $$ 
$$ \phi (c_{i} ) = e_{1}^{-1} ;\text{i is even and} \hspace{0.2cm} 5 \leq i \leq a_{4},$$ 
when \emph{ $ a_{4} $ is odd }
$$ \phi (c_{1} ) = e_{1},~ \phi (c_{2} ) = e_{2} $$ 
$$ \phi (c_{3} ) = e_{3},~ \phi (c_{4} ) = e_{1}^{-1},~ \phi (c_{5} ) = e_{1}, $$
$$ \phi (c_{i} ) = e_{1}^{-1} ;\text{i is odd and} \hspace{0.2cm} 6 \leq i \leq a_{4}, $$ 
$$ \phi (c_{i} ) = e_{1} ;\text{i is even and} \hspace{0.2cm} 6\leq i \leq a_{4}.$$
In the above $ \phi $ is surface-kernel epimorphism. Hence $ (0; 4^{[a_{4}]} ) $ is a signature of $ PSL( 2,7 ). $
\item[Subcase 2] $a_7\geq1$ \\
Now $ (0;2^{[a_{2}]},3^{[a_{3}]}, 4^{[a_{4}]}, 7^{[a_{7}]} ) $ is always a signature for $ PSL_2(\mathbb{F}_7) $ as the only maximal subgroup whose order is divisible by 7 is $ \mathbb{Z}_7 \rtimes \mathbb{Z}_3 $ but $  4 \nmid ~|\mathbb{Z}_7 \rtimes \mathbb{Z}_3|. $
\item[Subcase 3]$a_7=0$\\
$ a_{2}+a_{3}+a_{4} \geq 4. $
Consider $ (0; 2^{[a_{2}]},3^{[a_{3}]},4^{[a_{4}]}). $ Now $$ C_{2^{[a_{2}]},3^{[a_{3}]},4^{[a_{4}]}} ^{S_{4}} \leq 6^{a_{2}+a_{4}} \cdot 8^{a_{3}} / 12 \text{ and } C_{2^{[a_{2}]},3^{[a_{3}]},4^{[a_{4}]}}^{PSL_2(\mathbb{F}_7)} \geq 21^{a_{2}}  56^{a_{3}}  42^{a_{4}} / 144. $$
\noindent The only maximal subgroup in $PSL_2(\mathbb{F}_7)$ containg element of order $2,3,4$ is $S_4$, and $ C_{2^{[a_{2}]},3^{[a_{3}]},4^{[a_{4}]}}^{PSL_2(\mathbb{F}_7)} \geq C_{2^{[a_{2}]},3^{[a_{3}]},4^{[a_{4}]}} ^{S_{4}} . 7. $\\
\noindent So $ (0; 2^{[a_{2}]},3^{[a_{3}]},4^{[a_{4}]}) $ is a signature of $ PSL_2(\mathbb{F}_7). $
\end{itemize}
\item[Case 2] $a_4=0$
\begin{itemize}
\item[Subcase 1] $a_3\geq1 \text{ and }a_2,a_7=0$\\
\par
For $ a_{3} \geq 4,  \ we \ have \ g \geq 57 > 3. $ Now $ C^{S_{4}} _{3^{[4]}} \leq 427 $ , $ C^{\mathbb{Z}_7 \rtimes \mathbb{Z}_3} _{3^{[4]}} \leq 27 $ and  $ C^{PSL_2(\mathbb{F}_7)} _{3^{[4]}} \geq 74218 .$ But $ 7 \cdot 427 \leq 74218 $ and $ 8 \cdot 27 \leq 74218. $ Hence $ (0; 3^{[4]}) $ is a signature of  $ PSL_2(\mathbb{F}_7). $ Now by extension principle $ (0; 3^{[a_{3}]})_{a_{3} \geq 4} $ is a signature of  $ PSL_2(\mathbb{F}_7). $
\item[Subcase 2] $a_7\geq1 \text{ and }a_2,a_3=0$\\
\par
Since $ (0; 7,7,7) $ is a signature of $ PSL_2(\mathbb{F}_7),$ so by extension principle $ (0;7^{[a_{7}]} )_{a_{7} \geq 3} $ is always signature of $ PSL_2(\mathbb{F}_7). $
\item[Subcase 3] $a_7,~a_2\geq ~1 \text{ and } a_3=0$\\
\par
Since $(0;2,7,7)$ is a signature of $PSL_2(\mathbb{F}_7)$. So by extension principle $(0; 2^{[a_{2}]},7^{[a_{7}]}) $ is always a signature of $ PSL_2(\mathbb{F}_7)$ whenever $1-168+42a_2+72a_7\geq 3$. 
\item[Subcase 4]$a_7,a_3\geq1 \text{ and } a_2=0$\\
As $ (0;3,3,7) $ and $ (0; 3,7,7) $ are signature so by extension principle $ (0; 3^{[a_{3}]},7^{[a_{7}]} ) $ is always a signature of $ PSL_2(\mathbb{F}_7). $
\item[Subcase 5] $a_2,a_3\geq1 \text{ and } a_7=0$\\

\end{itemize}
As $(0;2,3,7)$ is a signature of $ PSL_2(\mathbb{F}_7). $ So $ \exists $ a surface-kernel epimorphism $ \phi: \Gamma(0;2,3,7) \rightarrow  PSL_2(\mathbb{F}_7).  $ where $ \Gamma(0;2,3,7)= \langle c_{1}, c_{2}, c_{3} \mid c_{1}^{2} =c_{2}^{3}=c_{3}^{7}=1,c_{1}c_{2}c_{3}=1 \rangle $ and $ \phi(c_{i}) = e_{i};i=1,2,3.$
Now we define a map
 $$ \phi: \Gamma(0; 2^{[a_{2}]}, 3^{[a_{3}]} )_{a_{2},a_{3} \ is \ even}  \rightarrow PSL_2(\mathbb{F}_7) \text{ as } $$  
$$ \phi(c_{1i})=e_{1i} ;~1 \leq i \leq a_{2} \text{ and i is odd,} $$ 
$$ \phi(c_{1i})=e_{1i}^{-1} ;~1 \leq i \leq a_{2} \text{ and i is even,} $$
$$ \phi(c_{2i})=e_{2i} ;~1 \leq i \leq a_{3} \text{ and i is odd,}$$ 
$$ \phi(c_{2i})=e_{2i}^{-1} ;~1 \leq i \leq a_{3}\text{ and i is even,} $$  
where 
$$ \Gamma(0; 2^{[a_{2}]}, 3^{[a_{3}]} )= \langle c_{11}, \dots , c_{1a_{2}}, c_{21}, \dots , c_{2a_{3}} \mid c_{11}^{2}
= \dots = c_{1a_{2}}^{2}=$$
 $$ c_{21}^{3}=c_{2a_{3}}^{3}=1,c_{11} \dots c_{1a_{2}} \dots c_{21} \dots c_{2a_{3}}=1 \rangle .  $$ 
Clearly, $ \phi: \Gamma(0; 2^{[a_{2}]}, 3^{[a_{3}]} ) \rightarrow PSL_2(\mathbb{F}_7) $ is surface-kernel epimorphism.
Now consider $ (0;2,2,2,3) $ and $ (0;2,3,3,3). $ If they do not generate $ PSL_2(\mathbb{F}_7) $ then the only maximal subgroup whose order is divided by both 2 and 3 is $ S_{4}. $ But an order 2 permutation is odd  permutation and order 3  permutation is even  permutation. That will lead to a contradiction. So $ (0;2,2,2,3) $ and $ (0;2,3,3,3) $ generate $ PSL_2(\mathbb{F}_7). $ Now by extension principle $ (0;2^{[a_{2}]},3^{[a_{3}]}) $ is a signature of $ PSL_2(\mathbb{F}_7) $ when $ a_{2}+ a_{3} \geq 4 $.
\end{itemize}

\noindent Now consider $ a_{2} \geq 5$ and $a_3,a_4,a_7=0$\\
For $ a_{2} \geq 5  \text{ we have } g \geq 43 > 3. $
Now $ 14.C^{S_{4}} _{2^{[5]}} = 0 $ or $78624$. Now $ C^{PSL_2(\mathbb{F}_7)} _{2^{[5]}}  \geq 134284. $
So $ C^{S_{4}} _{2^{[5]}} \cdot 14 \leq C^{PSL_2(\mathbb{F}_7)} _{2^{[5]}}.$ Let $H~= \langle e_{21},e_{22},e_{23},e_{24},e_{25}|e_{21}^2=~e_{22}^2=~e_{23}^2=~e_{24}^2=~e_{25}^2=$     $$~e_{21}.e_{22}.e_{23}.e_{24}.e_{25}=1;~e_{2i}\in PSL_2(\mathbb{F}_7),~ i=1,2,3,4,5\rangle,$$ then $H$ is a subgroup of $PSL_2(\mathbb{F}_7)$. If $H \neq PSL_2(\mathbb{F}_7)$, then $H$ is a subgroup of $S^4$, as $S^4$ is the only maximal subgroup that contains element of order $2$ in $PSL_2(\mathbb{F}_7)$. This implies all the solutions of the equation $e_{21}.e_{22}.e_{23}.e_{24}.e_{25}=1$ lie in $S_4$, as $S_4$ is the only maximal subgroup of $PSL_2(\mathbb{F}_7)$ that contains element of order 2, and  $e_{21},e_{22},e_{23},e_{24},e_{25} \in PSL_2(\mathbb{F}_7)$. This contradicts the fact $ C^{S_{4}} _{2^{[5]}} . 14 \leq C^{PSL_2(\mathbb{F}_7)} _{2^{[5]}}.$ \\ 
\noindent  Hence $ (0;2^{[5]}) $ is a signature of $ PSL_2(\mathbb{F}_7)$.
So by extension principle $ (0;2^{[a_{2}]}) $ is a signature of $ PSL_2(\mathbb{F}_7)$ for $a_2 \geq 5.$

\end{proof}

\noindent It is well known that $PSL_2(\mathbb{F}_7) $ is isomorphic to $GL_3(\mathbb{F}_2)$ \cite{isomat}, and along with that $GL_3(\mathbb{F}_2)$ is isomorphic to $PSL_3(\mathbb{F}_2)$. Let $ \psi: PSL_3(\mathbb{F}_2) \rightarrow PSL_2(\mathbb{F}_7) $ be an isomorphism. Now $ PSL_3(\mathbb{F}_2) $ is generated by \\
\[
A^{\shortmid} =
\begin{bmatrix}
1&1&0 \\
0&1&0 \\
0&0&1
\end{bmatrix}
,B^{\shortmid} =
\begin{bmatrix}
0&0&1 \\
1&0&0 \\
0&1&0
\end{bmatrix}\]
 Now    
$C^{\shortmid} =[A^{\shortmid},B^{\shortmid}]$ and $ |C^{\shortmid}| =4 $,
$ |A^{\shortmid}B^{\shortmid}| =7 $
, $ |B^{\shortmid}A^{\shortmid}| =7 $, $D^{\shortmid} =(A^{\shortmid}B^{\shortmid})^{-1}$, $[D^{\shortmid},B^{\shortmid}]= E^{\shortmid} \text{ and }  |E^{\shortmid}|=4.$ $[C^{\shortmid},E^{\shortmid}]= F^{\shortmid}$ and $ |F^{\shortmid}| = 7.$ $ B^{\shortmid}A^{\shortmid} (B^{\shortmid})^{-1} = G. $ $[F^{\shortmid},G^{\shortmid}]= H^{\shortmid},$ and $|H^{\shortmid}|=3.$ 
Let $$ \psi (A^{\shortmid}) = A, ~  \psi (B^{\shortmid}) = B, ~ \psi (C^{\shortmid}) = C, ~ \psi (D^{\shortmid}) = D, ~ \psi (E^{\shortmid}) = E, ~ \psi (F^{\shortmid}) = F, ~ \psi (G^{\shortmid}) = G,$$ $\psi (H^{\shortmid}) = H.$

\begin{lemma}
\label{lemma2}
$ ( 1;2^{[a_{2}]}, 3^{[a_{3}]}, 4^{[a_{4}]}, 7^{[a_{7}]} ) $ is a signature of $ PSL_2(\mathbb{F}_7) $ if and only if $ a_{2} + a_{3} + a_{4} + a_{7} \geq 1 $ and $ a_{i} \geq 0 $ for $ i = 2,3,4,7 $ except $ (1;2). $
\end{lemma}
\begin{proof}
$(1;2^{[a_{2}]}, 3^{[a_{3}]}, 4^{[a_{4}]}, 7^{[a_{7}]})$ is a signature whenever $(0;2^{[a_{2}]}, 3^{[a_{3}]}, 4^{[a_{4}]}, 7^{[a_{7}]})$ is a signature of $ PSL_2(\mathbb{F}_7).$ We can do it easily by extending the surface-kernel epimorphism map $\phi: \Gamma (0;2^{[a_{2}]}, 3^{[a_{3}]}, 4^{[a_{4}]}, 7^{[a_{7}]}) \rightarrow PSL_2(\mathbb{F}_7)$ to $\phi^{\shortmid}:\Gamma(1;2^{[a_{2}]}, 3^{[a_{3}]}, 4^{[a_{4}]}, 7^{[a_{7}]}) \rightarrow PSL_2(\mathbb{F}_7)$ by assigning the value $A$ for  the hyperbolic elements under the extension map $\phi^{\shortmid}$ and the values of elliptic element under $\phi^{\shortmid}$ is same as the value of elliptic element under $\phi$.\\
We have to only consider the following cases.
\begin{itemize}
\item[Case 1] $a_2\geq 1$
\begin{itemize}
\item[Subcase 1] $a_2\geq 2$ and $a_3,a_4,a_7=0.$
First consider $\phi: \Gamma(1;2,2) \rightarrow PSL_2(\mathbb{F}_7) \text{ as }$ 
$$\phi(\alpha_1)=B,~\phi(\beta_1)=1,~\phi(c_{21})=A,~\phi(c_{22})=A.$$
Then clearly $\phi$ is surface-kernel epimorphism. Here  $  \Gamma (1; 2,2)= \langle \alpha_{1},\beta_{1},c_{21},c_{22} \mid [\alpha_{1},\beta_{1}]c_{21}c_{22}=1,c_{21}^{2}=1=c_{22}^{2}  \rangle. $ So
$ (1; 2^{[a_2]}) $ will generate $ PSL_2(\mathbb{F}_7)$ by extension principle for $a_2\geq 2.$ 
\item[Subcase 2] $a_3\geq 1$ and $a_4,a_7=0$\\
 As  $ |[A,B]| =4 .$ So $ [A,B] \in S_{4}. $ Now any element of order $4$ is product of one element of order $2$ and one element of order $3$ in $S_4$. Let $ [A,B]^{-1} = \sigma_{2}\sigma_{3} $ where $ |\sigma_{2}| =2, |\sigma_{3}| =3. $ Now we define
 $$ \phi; \Gamma (1; 2,3) \rightarrow PSL_2(\mathbb{F}_7) \text{ as }$$ 
 $$  \phi(\alpha_{1})=A,~ \phi(\beta_{1})=B,~
  \phi(c_{2})= \sigma_{2},~
\phi(c_{3})= \sigma_{3} $$ where $  \Gamma (1; 2,3)= \langle \alpha_{1},\beta_{1},c_{2},c_{3} \mid [\alpha_{1},\beta_{1}]c_{2}c_{3}=1,c_{2}^{2}=1=c_{3}^{3}  \rangle. $ \\
Now by extension principle $ (1; 2^{[a_{2}]},3^{[a_{3}]}) $ is a signature of $ PSL_2(\mathbb{F}_7).$
\item[Subcase 3] $a_4\geq 1$ and $a_3,a_7=0$\\
So $ a_{4},a_{2} \geq 1. $ First consider $ (1;2,4). $ We define $$ \phi: \Gamma(1;2,4) \rightarrow PSL_2(\mathbb{F}_7) $$ as
$$ \phi(\alpha_{1}) = A,~
\phi(\beta_{1}) = B,~
\phi(c_{21}) = C^{2},~
\phi(c_{41}) = C . $$ Clearly $ \phi $ is surface-kernel epimorphism.\\
Consider the signature $(1;2,4,4).$ Now $(0;2,4,7)$ is a signature of $PSL_2(\mathbb{F}_7).$ So $\exists$ a surface-kernel epimorphism $\phi:\Gamma(0;2,4,7) \rightarrow PSL_2(\mathbb{F}_7 \text{ as, }~ \phi(c_{21})=e_{21},~\phi(c_{41})=e_{41},~\phi(c_{71})=e_{71}.$ Here $  \Gamma (0; 2,4,7)= \langle c_{21},c_{41},c_{71} \mid c_{21}c_{41}c_{71}=1,c_{21}^{2}=1=c_{41}^{4}=c_{71}^7  \rangle. $  We define $$ \phi: \Gamma(1;2,4,4) \rightarrow PSL_2(\mathbb{F}_7) $$ as
$$ \phi(\alpha_{1}) = 1,~
\phi(\beta_{1}) = e_{21},~
\phi(c_{21}) = e_{41}^{2},~
\phi(c_{41}) = e_{41},~
\phi(c_{42}) = e_{41} . $$ Clearly $ \phi $ is surface-kernel epimorphism.\\
From the above argument we can say now that $ (1;2^{[a_{2}]},4^{[a_{4}]}) $ is a signature of $ PSL_2(\mathbb{F}_7).$
\item[Subcase 4] $a_7\geq 1$ and $a_3,a_4=0$\\
Consider $$ \phi: \Gamma(1;2,7) \rightarrow PSL_2(\mathbb{F}_7) $$ as
$$ \phi(\alpha_{1}) = F
,~\phi(\beta_{1}) = G
,~\phi(c_{21}) = A, $$
\[
\phi(c_{71}) =\psi \left(
\begin{bmatrix}
0&0&1 \\
0&1&1 \\
1&1&1
\end{bmatrix} \right)
.\] \\
Clearly  $ \phi $ is surface-kernel epimorphism. Now by extension principle $ (1;2^{[a_{2}]},7^{[a_{7}]}) $ is a signature of $ PSL_2(\mathbb{F}_7). $
\item[Subcase 5] $a_3,a_4\geq 1$ and $a_7=0$\\
Now $ \Gamma (1; 2^{[a_{2}]},3^{[a_{3}]},4^{[a_{4}]}) = \langle \alpha_{1},\beta_{1},c_{21}, \dots , c_{2a_{2}},c_{31},\dots,c_{3a_{3}},c_{41},\dots,c_{4a_{4}} \mid$ $$[\alpha_{1},\beta_{1}] c_{21} \dots  c_{2a_{2}} c_{31} \dots c_{3a_{3}} \dots c_{41} c_{4a_{4}}=1,$$ $$ c_{21}^{2} = \dots = c_{2a_{2}}^{2} = \dots = c_{31}^{3} = \dots = c_{3a_{3}}^{3} = c_{41}^{4} = \dots = c_{4a_{4}}^{4} = 1   \rangle. $$ \\
Now we first prove $ (1;2,3,4) $ is a signature.
$ \Gamma (1; 2,3,4) = \langle \alpha_{1},\beta_{1},c_{2},c_{3},c_{4}\mid$ $$ [\alpha_{1},\beta_{1}]c_{2}c_{3}c_{4}=1, c_{2}^{2}=c_{3}^{3}=c_{4}^{4}=1 \rangle.$$ Now we define $$ \phi: \Gamma (1; 2,3,4) \rightarrow PSL_2(\mathbb{F}_7) $$ as
$$ \phi (\alpha_{1})= A,~
,~\phi(\beta_{1}) = B
,~\phi (c_{2})= A
,~ \phi (c_{3})= B,  $$
\[
\phi(c_{4})=\psi \left(
\begin{bmatrix}
1&1&0 \\
1&1&1 \\
0&1&1
\end{bmatrix}\right)
.\]
Clearly $ \phi: \Gamma(1;2,3,4) \rightarrow PSL_2(\mathbb{F}_7) $ is surface-kernel epimorphism.
Now consider $ (1;2,3, 4,4 ). $
Now define $ \phi: \Gamma(1;2,3, 4,4 ) \rightarrow PSL_2(\mathbb{F}_7) $ as
$$ \phi (\alpha_{1})= A,~\phi(\beta_{1}) = B$$ $$\phi (c_{2})= \sigma_{2}; |\sigma_{2}|=2, \sigma_{2} \in S_{4} $$
$$ \phi (c_{3})= \sigma_{3}; |\sigma_{3}|=3, \sigma_{3} \in S_{4} $$
$$ \phi(c_{41})= [A,B],~ \phi(c_{42})= [A,B]^{-1}. $$ \\
Clearly $ \phi: \Gamma(1;2,3, 4,4 ) \rightarrow PSL_2(\mathbb{F}_7) $ is surface-kernel epimorphism. Using above technique and extension principle we can say $ (1; 2^{[a_{2}]},3^{[a_{3}]},4^{[a_{4}]}) $ is a signature of $ PSL_2(\mathbb{F}_7). $
\end{itemize}
\item[Case 2] $a_2=0$
\begin{itemize}
\item[Subcase 1] $a_3\geq 1$ and $a_4,a_7=0$\\
\emph{Let $ a_{3} $ be odd}. Now consider $ (1;3^{[a_{3}]})_{a_{3} \geq 1}. $ Now we define 
$$ \phi: \Gamma (1;3^{[a_{3}]}) \rightarrow PSL_2(\mathbb{F}_7), \text{ as } $$  
$$ \phi(\alpha_{1}) = F, ~
 \phi(\beta_{1}) = G, $$
$$ \phi (c_{31})= H^{-1} $$
$$ \phi(c_{3i}) = H;~ 2 \leq i \leq a_{3} \text{ and i is even,} $$ 
$$ \phi(c_{3i}) = H^{-1} ;~ 2 \leq i \leq a_{3} \text{ and i is odd.} $$  
If $ \phi $ is not surface-kernel epimorphism then the only subgroup whose order is divisible by $7$ is $7:3$ but $2$ does not divide $7:3$. So $ \phi $ is surface-kernel epimorphism. Hence $ (1;3^{[a_{3}]})_{a_{3} \geq 1}. $ generate $ PSL_2(\mathbb{F}_7). $ \\
\emph{Let $ a_{3} $ be even.} \\
Now $ \Gamma (1;3^{[a_{3}]}) = \langle \alpha_{1},\beta_{1}, c_{31}, \dots, c_{3a_{3}} \mid [\alpha_{1},\beta_{1}]c_{31} \dots c_{3a_{3}} =1, |c_{31}| = \dots = |c_{3a_{3}}| = 3 \rangle $
As $ (0; 3,3,7) $ is a signature of $ PSL_2(\mathbb{F}_7) $ so we have surface-kernel epimorphism $$ \phi: \Gamma(0;3,3,7) \rightarrow PSL_2(\mathbb{F}_7) $$ where $ \Gamma(0,3,3,7)= \langle c_{1},c_{2},c_{3} \mid c_{1} c_{2} c_{3}=1, c_{1}^{3}=c_{2}^{3}=c_{3}^{7}=1 \rangle $ defined as $$ \phi(c_{i}) = e_{i} \text{ for i=1,2,3.} $$ \\
Now we define $ \phi: \Gamma (1;3^{[a_{3}]}) \rightarrow PSL_2(\mathbb{F}_7)  $ as 
$$ \phi(\alpha_{1})=e_{3},~
\phi(\beta_{1})=1,~
\phi(c_{31})=e_{1},~
\phi(c_{32})=e_{1}^{-1} $$
$$ \phi(c_{3i})=e_{2} ;~3 \leq i \leq a_{3}  \text{ and i is odd},$$  
$$ \phi(c_{3i})=e_{2}^{-1};~ 4 \leq i \leq a_{3} \text{ and i is even.}$$ 
Clearly $ \phi $ is surface-kernel epimorphism. So $  (1;3^{[a_{3}]}) $ generate $ PSL_2(\mathbb{F}_7). $
\item[Subcase 2] $a_4\geq 1$ and $a_3,a_7=0$\\
First we show that $ (1;4) $ is signature of $ PSL_2(\mathbb{F}_7). $ Consider a map $$ \phi: \Gamma(1;4) \rightarrow PSL_2(\mathbb{F}_7) $$ defined as 
$$ \phi(\alpha_{1}) = A,~ \phi(\beta_{1}) = B ,~ \phi(c_{1}) = C^{-1}. $$ Then clearly $ \phi $ is surface-kernel epimorphism.\\
Consider the signature $ (1;4,4). $  Now we define $$ \phi: \Gamma(1;4,4) \rightarrow PSL_2(\mathbb{F}_7)\text{ as }  $$  
$$ \phi(\alpha_{1}) = A,~
\phi(\beta_{1}) = B,~
,~\phi(c_{41}) = E,~
\phi(c_{42}) = (CE)^{-1}. $$ Then $ \phi $ is surface-kernel epimorphism. So clearly $ (1;4^{[a_{4}]})_{a_{4}\geq 1} $ is a signature of $ PSL_2(\mathbb{F}_7). $
\item[Subcase 3] $a_7\geq 1$ and $a_3,a_4=0$\\
Consider $ a_{7} \geq 1. $ We now define $$ \phi: \Gamma(1;7) \rightarrow PSL_2(\mathbb{F}_7) $$ as 
$$ \phi(\alpha_{1}) = C,~
\phi(\beta_{1}) = E,~
\phi(c_{71}) = F^{-1}. $$ Clearly $ \phi $ is surface-kernel epimorphism as only maximal subgroup whose order is divisible by 7 is $ \mathbb{Z}_7 \rtimes \mathbb{Z}_3 $ but 4 does not divide the order of $ \mathbb{Z}_7 \rtimes \mathbb{Z}_3 $. Now by extension principle $ (1;7^{[a_{7}]})_{a_{7}\geq 1} $ is a signature of $ PSL_2(\mathbb{F}_7).$
\item[Subcase 4] $a_3,a_4\geq1$ and $a_7=0$\\
Consider $ (1;3,4). $ We define $$ \phi: \Gamma(1;3,4) \rightarrow PSL_2(\mathbb{F}_7)\text{ defined as } $$ 
 $$ \phi(\alpha_{1})= D,~
\phi(\beta_{1})= B,~
\phi(c_{31})= B, $$
\[
\phi(c_{41}) =\psi \left(
\begin{bmatrix}
1&1&1 \\
1&1&0 \\
1&0&1
\end{bmatrix} \right)
.\] 
Consider $ (1;3,4,4). $ We know $ (0;3,4,4) $ is a signature of $ PSL_2(\mathbb{F}_7). $ Now $$ \Gamma (0;3,4,4) = \langle c_{31},c_{41},c_{42} \mid  c_{31}^{3}=c_{41}^{4}=c_{42}^{4}=1, c_{31}c_{41}c_{42}=1 \rangle. $$
 So $ \exists  $ a surface-kernel epimorphism $$ \phi: \Gamma (0;3,4,4) \rightarrow PSL_2(\mathbb{F}_7) $$ defined as
$$ \phi(c_{31})=e_{31},~ \phi(c_{41})=e_{41},~ \phi(c_{42})=e_{42}. $$
Now we define $$  \phi: \Gamma (1;3,4,4) \rightarrow PSL_2(\mathbb{F}_7) $$ as
$$ \phi(\alpha_{1})=e,~
\phi(\beta_{1})=A,~
,~\phi(c_{31})=e_{31}
,~\phi(c_{41})=e_{41}
,~\phi(c_{42})=e_{42}. $$ Clearly $$  \phi: \Gamma (1;3,4,4) \rightarrow PSL_2(\mathbb{F}_7) $$ is surface-kernel epimorphisms. So $ (1;3^{[a_{3}]},4^{[a_{4}]}) $ is signature of $ PSL_2(\mathbb{F}_7) $
\item[Subcase 5] $a_3,a_7\geq1$ and $a_4=0$\\
 Consider $ (1;3,7). $ We define a map $$ \phi: \Gamma(1;3,7) \rightarrow PSL_2(\mathbb{F}_7) $$ as
$$ \phi(\alpha_{1}) = F
,~\phi(\beta_{1}) = G
,~\phi(c_{31}) = B, $$
\[
\phi(c_{71}) =\psi \left(
\begin{bmatrix}
0&1&1 \\
1&1&1 \\
0&1&0
\end{bmatrix} \right)
.\] \\
Clearly  $ \phi $ is surface-kernel epimorphism. Now by extension principle $ (1;3^{[a_{3}]},7^{[a_{7}]}) $ is a signature of $ PSL_2(\mathbb{F}_7). $
\item[Subcase 6] $a_4,a_7\geq1$ and $a_3=0$\\
Now consider $ (1;4,7). $ Let $ \phi: \Gamma (1;4,7) \rightarrow PSL_2(\mathbb{F}_7) $ defined as
$$ \phi(a_{1}) = F,~ \phi(b_{1}) = E, $$
\[
\phi(c_{41}) =\psi \left(
\begin{bmatrix}
0&0&1 \\
0&1&1 \\
1&0&0
\end{bmatrix} \right)
,\] 
$$ \phi(c_{71}) = F. $$ If possible let $ \phi $ is not surface-kernel epimorphism. $ Im(\phi) $ contain a order 7 element and only maximal subgroup whose order is divisible by 7 is $ \mathbb{Z}_7 \rtimes \mathbb{Z}_3, \text{ and }  Im(\phi) $ also contain order 4 element and $ 4 \nmid |\mathbb{Z}_7 \rtimes \mathbb{Z}_3| $. A contradiction. so $ \phi $ is surface-kernel epimorphism.\\
Now consider $ (1;4,4,7). $ We know $ (0;4,4,7) $ is signature of $ PSL_2(\mathbb{F}_7). $ So $ \exists $ a surface-kernel epimorphism $$ \phi: \Gamma(0;4,4,7) \rightarrow PSL_2(\mathbb{F}_7) $$ defined as
$$ \phi(c_{41})=e_{41},~\phi(c_{42})=e_{42},~,~\phi(c_{71})=e_{71}.$$
Now we define $$ \overline{\phi}: \Gamma (1;4,4,7) \rightarrow PSL_2(\mathbb{F}_7). $$
as 
$$ \overline{\phi}(\alpha_{1}) =1,~
\overline{\phi}(\beta_{1}) =1
,~\overline{\phi}(c_{41}) =e_{41}
,~\overline{\phi}(c_{42}) =e_{42}
,~\overline{\phi}(c_{71}) =e_{71}. $$ Clearly $ \overline{\phi} $ define a surface-kernel epimorphism. So clearly $ (1;4^{[a_{4}]},7^{[a_{7}]}) $ is clearly signature using above technique.
\end{itemize}
\item[Case 3] Now we consider $(1;2).$ It is not signature for the group $PSL_2(\mathbb{F}_7)$ \cite{notsig,GAP,jen}. 
\end{itemize}  
\end{proof}

\begin{lemma}
\label{lemma3}
$ ( h;2^{[a_{2}]}, 3^{[a_{3}]}, 4^{[a_{4}]}, 7^{[a_{7}]} ) $ where $ h \geq 2 $ is a signature of $ PSL_2(\mathbb{F}_7) $ if and only if $ a_{i} \geq 0 $ for $ i = 2,3,4,7. $
\end{lemma}
\begin{proof}
We have to discuss only two following cases.
\begin{enumerate}
\item[Case 1]
let $ a_{i} = 0 ;\ i= 2,3,4,7. $ \\
We define $ \phi: \Gamma(h;-) \rightarrow PSL_2(\mathbb{F}_7) $ as 
$$ \phi(\alpha_{1}) = A 
,~\phi(\beta_{1}) = 1
,~\phi(\alpha_{2}) = B 
,~ \phi(\beta_{2}) = 1 $$
$$\phi(\alpha_{i}) = A ; \ 3 \leq i \leq h $$
$$\phi(\beta_{i}) = A ; \ 3 \leq i \leq h.  $$ Clearly $ \phi $ is surface-kernel epimorphism. 
\item[Case 2]
Now consider $(h_{\geq2};2).$ 

We have to prove $(h_{\geq2};2)$ is a signature of $PSL_2(\mathbb{F}_7).$ Now consider a map 
$$ \phi: \Gamma(h;2) \rightarrow PSL_2(\mathbb{F}_7)$$ as $\phi(\alpha_1)=A,~ \phi(\beta_1)=B,\phi(\alpha_2)=A,~ \phi(\beta_2)=B,~ \phi(c_1)= C^2. $\\
Then $ \phi: \Gamma(h;2) \rightarrow PSL_2(\mathbb{F}_7)$ is a surface -kernel epimorphism. Hence $(h_{\geq2};2)$ is a signature of $PSL_2(\mathbb{F}_7).$
\end{enumerate}
The other cases can be proved by the same method that were used in \ref{lemma1} and \ref{lemma2}.

\end{proof}

\begin{proof}[Proof of Theorem \ref{main theorem 1}] The proof follows from \ref{lemma1}, \ref{lemma2}, \ref{lemma3}.
\end{proof}

\begin{corollary}
$\chi_2 \text{ and } \chi_6$ \cite{tb} are the only two irreducible characters $PSL_2(\mathbb{F}_7)$ that comes from the Riemann surfaces corresponding to the signatures  
$(0;2,3,7) \text{ and } (0;3,3,4)$ respectively. 
\end{corollary}

\begin{corollary} 
The stable upper genus value of $PSL_2(\mathbb{F}_7)$ is $399$.

\end{corollary}

\noindent To prove our claim that the stable upper genus value of $PSL_2(\mathbb{F}_7)$ is $399$, we first have to show that $PSL_2(\mathbb{F}_7)$ can not act on a compact, connected, orientable surface of genus $398$ preserving the orientation. If possible let $$ 398 ~ = ~ 1+168(h-1) + 42a_2 + 56a_3 + 63a_4 + 72a_7.$$ Then the value of $h$ could be at most $4$. Similarly the values of $a_i$ could be at most $14,~ 11,~ 9,~ 8$ for $i= ~ 2,~ 3,~ 4,~ 7$ respectively. So We will consider $$0 ~ \leq h ~ \leq 4$$ $$0 ~ \leq a_2 ~ \leq 14$$ $$0 ~ \leq a_3 ~ \leq 11$$ $$0 ~ \leq a_4 ~ \leq 9$$ $$0 ~ \leq a_7 ~ \leq 8,$$ and execute the following python code. 

\newpage

\lstset{language=Python}
\lstset{frame=lines}
\lstset{caption={$398$ is not an admissable signature of $PSL_2(\mathbb{F}_7)$}}
\lstset{label={lst:code_direct}}
\lstset{basicstyle=\footnotesize}
\begin{lstlisting}
def func1(h,a2,a3,a4,a7):
   return 1+168*(h-1) + 42*a2 + 56*a3 + 63*a4 + 72*a7


 for h in range(4):
   for a2 in range(14):
     for a3 in range(11):
       for a4 in range(9):
         for a7 in range(8):
           sol = func1(h,a2,a3,a4,a7)
           if sol >3:
             if sol < 400:
               if sol == 398:
                print("wrong") 


\end{lstlisting}

\noindent It is immediate from above that $PSL_2(\mathbb{F}_7)$ can not act on a compact, connected, orientable surface of genus $398$ preserving the orientation. It is well understood from the following table that $PSL_2(\mathbb{F}_7)$ can act on a compact, connected, orientable surface of genus $g$ where $399 ~ \leq ~ g ~ \leq ~ 441$ preserving the orientation. Now we have to prove that $PSL_2(\mathbb{F}_7)$ can act on all compact, connected, orientable surface of genus $g ~ \geq ~ 441$ preserving the orientation. Let $g ~ \geq 442$, and $\Sigma_{g}$ be a compact, connected, orientable surface of genus $g$. So we have $$  g-399 ~ \equiv ~ s ~ (mod ~42) ~ \text{ where } ~1 ~ \leq ~ s ~ \leq 41.$$ Then $g ~ = ~ l+n.42$ where $ l ~= 399+ s$. We know the signature corresponding to the genus $l$ as $399~\leq l~ \leq 441$ and let it be $(h;m_2,~m_3,~m_4,~m_7)$. Then the signature corresponding to the genus $g$ is $(h;m_2+n,~m_3,~m_4,~m_7)$. In this way we can find signature corresponding to genus $g ~ \geq 442$. This completes the proof of our claim.
 \newpage
 
 \begin{table}
\resizebox{8cm}{!}
	{\begin{tabular}{|c | c |c |c | c| c|}
		\hline
		Genus value & orbifold genus $= h$ & $a_2$& $a_3$ & $a_4$ & $a_7$ \\ \hline
		$399$ & $0$ & $4$ & $1$ & $2$ & $3$ \\  
		$400$ & $0$ & $0$ & $0$ & $1$ & $7$ \\
		$401$ & $0$ & $0$ & $5$ & $0$ & $4$ \\
		$402$ & $0$ & $0$ & $1$ & $7$ & $1$\\
		$403$ & $0$ & $1$ & $3$ & $0$ & $5$\\
		$404$ & $0$ & $0$ & $2$ & $5$ & $2$\\
		$405$ & $0$ & $2$ & $1$ & $0$ & $6$\\
		$406$ & $0$ & $0$ & $3$ & $3$ & $3$\\
		$407$ & $0$ & $0$ & $8$ & $2$ & $0$\\
		$408$ & $0$ & $0$ & $4$ & $1$ & $4$\\
		$409$ & $0$ & $0$ & $0$ & $0$ & $8$\\
		$410$ & $0$ & $1$ & $2$ & $1$ & $5$\\
		$411$ & $0$ & $0$ & $1$ & $6$ & $2$\\
		$412$ & $0$ & $2$ & $0$ & $1$ & $6$\\
		$413$ & $0$ & $0$ & $2$ & $4$ & $3$\\
		$414$ & $0$ & $0$ & $7$ & $3$ & $0$\\
		$415$ & $0$ & $0$ & $3$ & $2$ & $4$\\
		$416$ & $0$ & $0$ & $8$ & $1$ & $1$\\
		$417$ & $0$ & $0$ & $4$ & $0$ & $5$\\
	    $418$ & $0$ & $0$ & $0$ & $7$ & $2$\\
		$419$ & $0$ & $1$ & $2$ & $0$ & $6$\\
		$420$ & $0$ & $0$ & $1$ & $5$ & $3$\\
		$421$ & $0$ & $0$ & $6$ & $4$ & $0$\\
		$422$ & $0$ & $0$ & $2$ & $3$ & $4$\\
		$423$ & $0$ & $0$ & $7$ & $2$ & $1$\\
		$424$ & $0$ & $0$ & $3$ & $1$ & $5$\\
		$425$ & $0$ & $0$ & $8$ & $0$ & $2$\\
		$426$ & $0$ & $1$ & $1$ & $1$ & $6$\\
		$427$ & $0$ & $0$ & $0$ & $6$ & $3$\\
		$428$ & $0$ & $0$ & $5$ & $5$ & $0$\\
		$429$ & $0$ & $0$ & $1$ & $4$ & $4$\\
		$430$ & $0$ & $0$ & $6$ & $3$ & $1$\\
		$431$ & $0$ & $0$ & $2$ & $2$ & $5$\\
		$432$ & $0$ & $0$ & $7$ & $1$ & $2$\\
		$433$ & $0$ & $0$ & $3$ & $0$ & $6$\\
		$434$ & $0$ & $1$ & $5$ & $1$ & $3$\\
		$435$ & $0$ & $0$ & $4$ & $6$ & $0$\\
		$436$ & $0$ & $0$ & $0$ & $5$ & $4$\\
		$437$ & $0$ & $0$ & $5$ & $4$ & $1$\\
		$438$ & $0$ & $0$ & $1$ & $3$ & $5$\\
		$439$ & $0$ & $0$ & $6$ & $2$ & $2$\\
		$440$ & $0$ & $0$ & $2$ & $1$ & $6$\\
		$441$ & $0$ & $0$ & $7$ & $0$ & 
		$3$\\ \hline
		 
	\end{tabular}}
	\caption{Signature values corresponding to genus values}
		\label{Table 1: Genus vs Signature}
	\end{table}

\newpage
\begin{corollary}
For a signature $(h;2^{[a_2]},3^{[a_3]},4^{[a_4]},7^{[a_7]})$ we have a branched covering $\Sigma_g$ of $\Sigma_h$ with $Deck(\Sigma_g/\Sigma_h)=PSL_2(\mathbb{F}_7)$ where $g=1+168(h-1)+42.a_2+56.a_3+63.a_4+72.a_7$ and $g \geq ~3.$ So for such every surface $\Sigma_g$ we will have a unique irreducible polynomial $F(w)$ over $\mathcal{M}(\Sigma_h)$ using $Theorem 2.6$ such that $\mathcal{M}(\Sigma_g) ~ \cong ~ \mathcal{M}(\Sigma_h)(F(w))$.  
\end{corollary}

\section{Signatures of $PSL_2(\mathbb{F}_{11})$}
\begin{lemma} \label{1}
$(0;2^{[a_{2}]},3^{[a_{3}]},5^{[a_{5}]},6^{[a_{6}]},11^{[a_{11}]})$ is a signature of $PSL_2(\mathbb{F}_{11})$ if and only if $-659+ 165a_{2} + 220a_{3} + 264a_{5} + 275a_6 +300a_{11} \geq 26$. 
\end{lemma}

\begin{proof}
We have to consider the following cases.
\begin{itemize}
\item[Case 1.] $a_{11}>0.$
\begin{itemize}
\item[Subcase 1.] Let $a_2>0,~a_3,~a_5,~a_6\geq 0.$ The class multiplication coefficient is always non zero. Also the only maximal subgroup which contain element of order $11$ is $\mathbb{Z}_{11} \rtimes \mathbb{Z}_5,$ but $2 \nmid |\mathbb{Z}_{11} \rtimes \mathbb{Z}_5|$. Hence $(0;2^{[a_{2}]},3^{[a_{3}]},5^{[a_{5}]},6^{[a_{6}]},11^{[a_{11}]})$ is always a signature of $PSL_2(\mathbb{F}_{11}).$
\item[Subcase 2.] Let $a_3 >0,~a_2,~a_5,~a_6\geq 0.$ By the similar argument as above $(0;2^{[a_{2}]},3^{[a_{3}]},5^{[a_{5}]},6^{[a_{6}]},11^{[a_{11}]})$ is again a signature of $PSL_2(\mathbb{F}_{11}).$
\item[Subcase 3.] $a_6 > 0,~ a_2,~ a_3,~ a_5 \geq0.$ By the similar argument as above it is clear that $(0;2^{[a_{2}]},3^{[a_{3}]},5^{[a_{5}]},6^{[a_{6}]},11^{[a_{11}]})$ is a signature of $PSL_2(\mathbb{F}_{11}).$
\item[Subcase 4.] Let $a_5 > 0,~ a_2,~ a_3,~ a_6 = 0.$ Now $$C_{5,5,11}^{\mathbb{Z}_{11} \rtimes \mathbb{Z}_5}=0=C_{5,11,11}^{\mathbb{Z}_{11} \rtimes \mathbb{Z}_5},$$ and $$C_{5,5,11}^{PSL_2(\mathbb{F}_{11})}=1320$$ $$C_{5,11,11}^{PSL_2(\mathbb{F}_{11})}=1440.$$ Hence $(0;5,5,11) \text{ and } (0;5,11,11)$ are signatures of $PSL_2(\mathbb{F}_{11}).$ As any element of order $5 \text{ and } 11$ can be written as product of two element of order $5 \text{ and } 11$ respectively. Then by extension principle  $(0;5^{[a_{5}]},11^{[a_{11}]})_{a_{11}\geq 2}$ and $(0;5^{[a_{5}]},11^{[a_{11}]})_{a_5 \geq 2}$ are signatures of $PSL_2(\mathbb{F}_{11}).$
\item[Subcase 5.] Now we consider $a_i=0$ for $i=2,3,5,6.$ If possible let $(0;11,11,11)$ is not a signature for $PSL_2(\mathbb{F}_{11}).$ The only maximal subgroup which contain element of order $11$ is $\mathbb{Z}_{11} \rtimes \mathbb{Z}_5$ and we have $12$ isomorphic copies of $\mathbb{Z}_{11} \rtimes \mathbb{Z}_5$ in $PSL_2(\mathbb{F}_{11}).$ Now $$C_{11A,11A,11A}^{PSL_2(\mathbb{F}_{11})}=840=C_{11B,11B,11B}^{PSL_2(\mathbb{F}_{11})},$$ and $$12\cdot C_{11A,11A,11A}^{\mathbb{Z}_{11} \rtimes \mathbb{Z}_5}=180=12 \cdot C_{11B,11B,11B}^{\mathbb{Z}_{11} \rtimes \mathbb{Z}_5},$$ $$12 \cdot C_{11A,11B,11B}^{\mathbb{Z}_{11} \rtimes \mathbb{Z}_5}=120=12 \cdot C_{11A,11A,11B}^{\mathbb{Z}_{11} \rtimes \mathbb{Z}_5}.$$ Hence $(0;11,11,11)$ is a signature of $PSL_2(\mathbb{F}_{11}).$ Now by extension principle $(0;11^{[a_{11}]})_{a_{11}\geq3}$ is a signature of $PSL_2(\mathbb{F}_{11}).$
\item[Subcase 6.] Now we assume $a_2=0,~a_3,~a_5,~a_6 > 0.$ If possible let $(0;3^{[a_3]},5^{[a_5]},6^{[a_6]},11^{[a_{11}]})$ is not a signature of $PSL_2(\mathbb{F}_{11}).$ Then $3\nmid 55$ as the only maximal subgroup containg element of order $11$ is $\mathbb{Z}_{11} \rtimes \mathbb{Z}_5$ which is not possible. Hence $(0;3^{[a_3]},5^{[a_5]},6^{[a_6]},11^{[a_{11}]})$ signature of $PSL_2(\mathbb{F}_{11}).$

In a similar way, we can show that $(0;2^{[a_2]},5^{[a_5]},6^{[a_6]},11^{[a_{11}]}),$ $(0;2^{[a_2]},3^{[a_3]},6^{[a_6]},11^{[a_{11}]})$,  $(0;2^{[a_2]},3^{[a_3]},5^{[a_5]},11^{[a_{11}]})$ are signatures of $PSL_2(\mathbb{F}_{11})$  when $a_3=0,~a_5=0,~a_6=0$ respectively.
\end{itemize}
\item[Case 2.] Now we assume $a_{11}=0.$
\begin{itemize}
\item[Subcase 1.] Let $a_2 \geq 1,~ a_3,~a_5,~a_6=0.$
Now $C^{PSL_2(\mathbb{F}_{11})}_{2^{[a_2]}}=6 \cdot 55^{a_{2}-2},$ but $22 \cdot C^{A_4}_{2^{[a_2]}}=154 \cdot 15^{a_2-2} \text{ and } 55 \cdot C^{D_{12}}_{2^{[a_2]}}= 0 \text{ or } \frac{55[2+5.2^{a_2}-2]}{6}.$ Clearly for  $a_2 \geq 5 ~ (0;2^{[a_2]})$ is a signature of $PSL_2(\mathbb{F}_{11}).$
\item[Subcase 2.] $a_3 \geq 1, ~ a_2,~a_5,~a_6 =0.$ Now $55.C^{D_{12}}_{3^{[a_3]}}=55.\frac{2^{a_2}[4+2.-1^{a_2}+2^{a_2-1}]}{12}$ and $22. C^{A_5}_{3^{[a_3]}}\leq 22.10. 20^{a_3-2}.$ Now $C^{PSL_2(\mathbb{F}_{11})}_{3^{[a_3]}}\geq 12.110^{a_3-2}.$ Clearly $(0;3^{[a_3]})$ is a signature for $PSL_2(\mathbb{F}_{11})$ whenever $a_3 \geq 4.$
\item[Subcase 3.] $a_5 \geq 1, ~ a_2,~a_3,~a_6 =0,$ and $a_6 \geq 1, ~ a_2,~a_3,~a_5 =0.$ By similar kind of arguement as abobe we can prove that $(0;5^{[a_5]})$ and $(0;6^{[a_6]})$ is a signature for $PSL_2(\mathbb{F}_{11})$ whenever $a_5,a_6\geq 3.$
\item[Subcase 4.] $a_2,~ a_3 \geq 1 \text{ and } a_5,~ a_6=0.$ Now $C_{2^{[3]},3^{[1]}}^{PSL_2(\mathbb{F}_{11})}\geq 15125$ Now $22.C_{2^{[3]},3^{[1]}}^{A_5}= 990,$ and $55.C_{2^{[3]},3^{[1]}}^{D_{12}} \leq 0.$ $C_{2^{[3]},3^{[1]}}^{PSL_2(\mathbb{F}_{11})}\geq 60500$ Now $22.C_{2^{[1]},3^{[3]}}^{A_5}=52800,$ and $55.C_{2^{[1]},3^{[3]}}^{D_{12}} \leq 3520.$ Consider $(0;2^{[2]},3^{[2]})$ where $a_2,~a_3$ are even. As $(0;2,3,11)$ is a signature of $PSL_2(\mathbb{F}_{11}),$ so we have a surface kernel homomorphisom $\phi:\Gamma(0;2,3,11) \rightarrow PSL_2(\mathbb{F}_{11})$ defined as $$\phi(c_2)=e_2,~\phi(c_3)=e_3,~\phi(c_{11})=e_{11}$$ where $$ \Gamma(0;2,3,11)= \langle c_{2}, c_{3}, c_{11} \mid c_{2}^{2} =c_{3}^{3}=c_{11}^{11}=1,c_{2}c_{3}c_{11}=1 \rangle.$$ Now we can define a map $\phi:\Gamma(0;2^{[a_2]},3^{[a_3]}) \rightarrow PSL_2(\mathbb{F}_{11})$ defined as \\ 
 
 $ \phi(c_{2i})=e_2;~ \text{i is odd },\\ 
\phi(c_{2i})=e_2^{-1};~ \text { i is even },\\
\phi(c_{3i})=e_3;~ \text{ i is odd },\\
\phi(c_{3i})=e_3^{-1};~ \text{ i is even } $ where $$\Gamma(0;2^{[a_2]},3^{[a_3]})= \langle c_{21},\cdots, c_{2a_2},c_{31}, \cdots, c_{3a_3}|c_{21}^2= \cdots = c_{2a_2}^2=1=c_{31}^3=$$ $$ \cdots = c_{3a_3}^3=1=c_{21}.\cdots c_{2a_2}c_{31}. \cdots c_{3a_3} \rangle.$$ Clearly $(0;2^{[a_2]},3^{[a_3]})$ is a signature of $PSL_2(\mathbb{F}_{11}).$ Hence by extension principle $(0;2^{[a_2]},3^{[a_3]})$ is a signature of $PSL_2(\mathbb{F}_{11}).$ 
\item[Subcase 5.] $a_2,~a_5 \geq 1$ and $a_5,a_6=0.$ Now $C^{PSL_2(\mathbb{F}_{11})}_{2^{[3]},5}=30250$ and $22.C^{A_5}_{2^{[3]},5}=9900.$ So $(0;2^{[3]},5)$ is a signature of $PSL_2(\mathbb{F}_{11}).$ Similarly $(0;2,5^{[3]})$ is a signature of $PSL_2(\mathbb{F}_{11}).$ As $(0;2,5,11)$ is a signature of $PSL_2(\mathbb{F}_{11}).$ By previous argument $(0;2^{[a_2]},5^{[a_5]})$ is a signature of $PSL_2(\mathbb{F}_{11})$ whenever $a_2,a_5$ are even. Hence $(0;2^{[a_2]},5^{[a_5]})$ is signature of $PSL_2(\mathbb{F}_{11}).$
\item[Subcase 6.] $a_2,a_6 \geq 1$ and $a_3,a_5=0.$ Now $C^{PSL_2(\mathbb{F}_{11})}_{2^{[a_2]},6^{[a_6]}}=55^{a_2-1}.110^{a_6-1}.12 $ and $C^{A_5}_{2^{[a_2]},6^{[a_6]}}\leq 0$ whenever $a_2 \text{ is odd and } a_6 \text{ is even }$.  $C^{PSL_2(\mathbb{F}_{11})}_{2^{[a_2]},6^{[a_6]}}=55^{a_2-1}.110^{a_6-1}.12$ and $C^{A_5}_{2^{[a_2]},6^{[a_6]}}\leq 220.6^{a_2-1}.2^{a_6-1}$ whenever $ a_6$ is odd and $a_2$ is even. Now we consider the case when $a_2,a_6$ is even, then $C^{PSL_2(\mathbb{F}_11)}_{2^{[a_2]},6^{[a_6]}}=55^{a_2-1}.110^{a_6-1}.18$ and $55.C^{A_5}_{2^{[a_2]},6^{[a_6]}}=55.\frac{2^{a_6[4+2^{a_2+1}+2^{a_2+a_6-1}]}}{12}$ or $.220.6^{a_2-1}2^{a_6-1}.$ If $a_2,~a_6$ are both odd then $C^{PSL_2(\mathbb{F}_11)}_{2^{[a_2]},6^{[a_6]}}=55^{a_2}.110^{a_6-2}.20$ and $55.C^{A_5}_{2^{[a_2]},6^{[a_6]}}=55.\frac{2^{a_6}[4-2^{a_2+1}+2^{a_2+a_6-1}]}{12}$ or $55.\frac{2^{a_6}[4-2^{a_2+1}+2^{a_2+a_6-1}]}{12}.$ So $(0;2^{[a_2]},6^{[a_6]})$ is signature of $PSL_2(\mathbb{F}_{11})$ whenever $a_2+a_6 \geq 3.$ 
\item[Subcase 7.] We will consider the case $a_3,a_5 \geq 1$ and $a_2,a_6=0.$ Now we have $C^{PSL_2(\mathbb{F}_{11})}_{3^{[2]},5^{[1]}}= 2640=C^{PSL_2(\mathbb{F}_{11})}_{3^{[1]},5^{[2]}}$ and $22.C^{A_5}_{3^{[2]},5^{[1]}}=1320=22.C^{A_5}_{3^{[1]},5^{[2]}}$. By extension theorem we can easily say that $(0;3^{[a_3]},5^{[a_5]})$ is a signature of $PSL_2(\mathbb{F}_{11}).$
\item[Subcase 8.] Now consider $a_3,a_6 \geq 1$ and $a_2,a_5=0.$ Now we have $C^{PSL_2(\mathbb{F}_{11})}_{3^{[a_3]},6^{[a_6]}}\geq 8.110^{a_3+a_2-2}$ and $55.C^{D_{12}}_{3^{[a_3]},6^{[a_6]}} \leq 5.2^{a_3+a_6}[6+2^{a_3+a_6-1}].$ Hence $(0;3^{[a_3]},6^{[a_6]})$ is a signature of $PSL_2(\mathbb{F}_{11})$ whenever $a_3+a_6 \geq 3.$
\item[Subcase 9.] Now we consider $a_5,a_6 \geq 1$ and $a_2,a_3=0.$ As $C^{PSL_2(\mathbb{F}_{11})}_{5^{[a_5]},6^{[a_6]}}=22.132^{a_5-1}.110^{a_6-1}.$ As there is no maximal subgroup of $PSL_2(\mathbb{F}_{11})$ containing element of orders $5 \text{ and } 6$ both. Hence $(0;5^{[a_5]},6^{[a_6]})$ is signature of $PSL_2(\mathbb{F}_{11})$ whenever $a_5+a_6 \geq 3.$
\item[Subcase 10.] Consider the case when $a_2,a_3,a_5 \geq 1, a_6=0.$ We can easily compute that $C_{2^{[a_2]},3^{[a_3]},5^{[a_5]}}^{PSL_2(\mathbb{F}_{11})}=20.55^{a_2}110^{a_3-1}132^{a_5-1}.$ The only maximal subgroup which contain elements of orders $2,3,5$ is $A_5$. Now  $C_{2^{[a_2]},3^{[a_3]},5^{[a_5]}}^{A_5}=22.15^{a_2}20^{a_3}12^{a_5}.$ Hence $(0;2^{[a_2]},3^{[a_3]},5^{[a_5]})$ is a signature of $PSL_2(\mathbb{F}_{11})$ whenever $165a_2+220a_3+264a_5-659 \geq 26.$
\item[Subcase 11.] Now we consider the case $a_2,a_3,a_6 \geq 1$ and $a_5=0.$ We have $C^{PSL_2(\mathbb{F}_{11})}_{2^{[a_2]},3^{[a_3]},6^{[a_6]}} \geq 55^{a_2-1}110^{a_3+a_6-1}.9$, and $55.C^{D_{12}}_{2^{[a_2]},3^{[a_3]},6^{[a_6]}} \leq 55.\frac{2^{a_3+a_6}[4+2^{a_2+1}+2^{a_2+a_3+a_6}]}{12} ~ \text{ or} ~ \\ 220.6^{a_2-1}.2^{a_3+a_6-1}$. Hence $(0;2^{[a_2]},3^{[a_3]},6^{[a_6]})$ is a signature of $PSL_2(\mathbb{F}_{11})$ a signature whenever $$165a_2+220a_3+275a_6-659 \geq 26.$$
\item[Subcase 12.] $a_3,a_5,a_6 \geq 1$ and $a_2=0.$ We have $$C^{PSL_2(\mathbb{F}_{11})}_{3^{[a_3]},5^{[a_5]},6^{[a_6]}} \geq 20. 110^{a_3+a_6-1}.132^{a_5-1}.$$ If possible let $(0;3^{[a_3]},5^{[a_5]},6^{[a_6]})$ is not signature whenever $220a_3+264a_5+275a_6-659 \geq 26$ then the only maximal subgroup which contain elements order of $6$ is $D_{12},$ but $5 \nmid 12,$ then it will lead to a contradiction. So $(0;3^{[a_3]},5^{[a_5]},6^{[a_6]})$ is a signature whenever $220a_3+264a_5+275a_6-659 \geq 26.$ 
\item[Subcase 13.] $a_2,a_5,a_6 \geq 1$ and $a_3=0.$ We have $$C^{PSL_2(\mathbb{F}_{11})}_{2^{[a_2]},5^{[a_5]},6^{[a_6]}} \geq 20.55^{a_2}. 110^{a_6-1}.132^{a_5-1}.$$ If possible let $(0;2^{[a_2]},5^{[a_5]},6^{[a_6]})$ is not signature whenever $220a_3+264a_5+275a_6-659 \geq 26$ then the only maximal subgroup which contain elements order $6$ is $D_{12},$ but $5 \nmid 12,$ then it will lead to a contradiction. So $(0;2^{[a_2]},5^{[a_5]},6^{[a_6]})$ is a signature whenever $165a_2+264a_5+275a_6-659 \geq 26.$ 
\end{itemize}
\end{itemize}

\end{proof}
\noindent The following result recites a relationship with an element g of a group $G$ and a commutator $[x,y]$ of the group $G$. The statement is as follows: \\
Let $G$ be a finite group and $g \in G.$ For a fix $x \in G,$ $g$ is conjugate to $[x,y]$ for some $y \in G$ if and only if $\displaystyle\sum_{\chi \in Irr(G)} \frac{|\chi(x)|^2. \overline{\chi(g)}}{\chi(1)} \neq 0$ \cite{char}.

The above result will play a crucial role to prove the next lemma.

\begin{lemma} \label{2}
$(1;2^{[a_2]},3^{[a_3]},5^{[a_5]},6^{[a_6]},11^{[a_{11}]})$ is signature of $PSL_2(\mathbb{F}_{11})$ if and only if atleast one $a_i \geq 1$ for $i=2,3,5,6,11.$
\end{lemma}

\begin{proof}

\begin{itemize}
\item[Case 1.] First we consider $g=2,~ x=11.$ Then we have $\displaystyle\sum_{\chi \in Irr(G)} \frac{|\chi(11)|^2. \overline{\chi(2)}}{\chi(1)} = \frac{11}{5}.$ Hence $(1;2)$ is a signature of $PSL_2(\mathbb{F}_{11})$ as only maximal sunbgroup of $PSL_2(\mathbb{F}_{11})$ containing element of order $11$ is $\mathbb{Z}_{11} \rtimes \mathbb{Z}_5$ but $2 \nmid 55.$
\item[Case 2.] Now we consider $g=5,~ x=6.$ Then we have $\displaystyle\sum_{\chi \in Irr(G)} \frac{|\chi(6)|^2. \overline{\chi(5)}}{\chi(1)} = \frac{12}{11}.$ Hence $(1;5)$ is a signature of $PSL_2(\mathbb{F}_{11})$ as only maximal sunbgroup of $PSL_2(\mathbb{F}_{11})$ containg element of order $6$ is $D_{12}$ but $5 \nmid 12.$ 

By similar kind of argument we can prove that $(1;6)$ and $(1;11)$ are a signature of $PSL_2(\mathbb{F})_{11}$ considering $g=6,~ x=5$ and $g=11,~x=6$ respectively. We have $$\displaystyle\sum_{\chi \in Irr(G)} \frac{|\chi(5)|^2. \overline{\chi(6)}}{\chi(1)} = \frac{10}{11},$$ and $$\displaystyle\sum_{\chi \in Irr(G)} \frac{|\chi(6)|^2. \overline{\chi(11)}}{\chi(1)} = \frac{3}{15}.$$
\item[Case 3.] We will discuss the case $a_3 \geq 1$ and $a_2,~a_5,~a_6 =0.$ For this particular we use GAP and did run the following code:

G := SmallGroup(660,13);\\
A := [];;\\
for x in G do\\
for y in G do\\
if Order( Comm( x,y ))=3\\
then\\
AddSet( A,[x,y] );\\
fi; od; od;\\
Size( A );\\
Using above code we get we get $|A|=51480.$ \\
To prove that $(1;3)$ is a signature we run the followin code in GAP \\
B := [];;\\
for a in A do\\
if G = SubGroup( G,a )\\
then\\
AddSet( B,a );\\
Size( B );\\
fi;od;\\  
We get $|B|=23760.$ \\
So $(1;3)$ is also a signature. By extension principle $(1;3^{[a_3]})$ is a signature of $PSL_2(\mathbb{F}_{11}).$
\end{itemize}

\noindent Now consider the $(1;2,3)$. As $(0;2,3,11)$ is a signature, so there exists a surface kernel epimorphisom $\phi ~ : ~ \Gamma (0;2,3,11) ~ \rightarrow ~ PSL_2(\mathbb{F}_{11})$ such that $\phi(c_{21})=e_{21},~ \phi(c_{31})=e_{31}, ~ \phi(c_{111})=e_{111}$ where $\Gamma(0;2,3,11)= ~ \langle c_{21},c_{31},c_{111}|c_{21}c_{31}c_{111}=1, ~ |c_{21}|=2,~ |c_{31}|=3,~ |c_{111}|=11 \rangle$. As $PSL_2(\mathbb{F}_{11})$ is non abelian, simple group. So the commutator subgroup $[PSL_2(\mathbb{F}_{11}),~ PSL_2(\mathbb{F}_{11})]=~ PSL_2(\mathbb{F}_{11})$. So $\exists ~ a,b ~ \in ~ PSL_2(\mathbb{F}_{11})$ such that $[a,b]= ~ e_{111}$. Now we can define a map $\psi ~ : \Gamma(1;2,3) ~ \rightarrow ~ PSL_2(\mathbb{F}_{11})$ as $$\psi(\alpha_1)~ = ~ a, ~ \psi(\beta_1)~ = ~ b, ~ \psi(c_{21}) ~ = ~ e_{21}, ~ \psi(c_{31}) ~ = ~ e_{31}$$ where $\Gamma(1;2,3)=\langle \alpha_1,\beta_1,c_{21},c_{31}| ~ [\alpha_1,\beta_1]c_{21}c_{31}=1, ~ |c_{21}|=2, ~ |c_{31}|=3\rangle$. Clearly $\psi$ is surface kernel epimorphisom. Hence $(1;2,3)$ is a signature of $PSL_2(\mathbb{F}_{11})$. Similarly we can prove that $(1;2,5),~ (1;2,6),~ (1;2,11),~ (1;3,5),~ (1;3,6),(1;3,11),(1;5,6),(1;5,11),(1;6,11),\\(1;11,11) $ are all signatures of $PSL_2(\mathbb{F}_{11})$.

Now using the previous $Lemma~ 4.2$ it is clear that $(1;2^{[a_2]},3^{[a_3]},5^{[a_5]},6^{[a_6]},11^{[a_{11}]})$ is a signature of $PSL_2(\mathbb{F}_{11})$ if and only if atleast one $a_i \geq 1$ for $i=2,3,5,6,11.$
\end{proof}

\begin{lemma}\label{3}
$(h;2^{[a_2]},3^{[a_3]},5^{[a_5]},6^{[a_6]},11^{[a_{11}]})_{h \geq 2}$ is a signature of $PSL_2(\mathbb{F}_{11})$ if and only if $a_i \geq 0$ for $i=2,~3,~5,~6,~11.$
\end{lemma}
\begin{proof}
We only have to prove the case $(h_{\geq ~ 2};-)$. We already see $(0;2,3,11)$ is a signature of $PSL_2(\mathbb{F}_{11})$. So we have a surface kernel epimorphism $$\phi:~ \Gamma(0;2,3,11) ~ \rightarrow ~ PSL_2(\mathbb{F}_{11}) ~ \text{defined as} $$ $$\phi(c_{21})=e_{21}~ \phi(c_{31})=e_{31} ~ \text{and} ~ \phi(c_{111})=e_{111}$$ where $\Gamma(0;2,3,11)=\langle c_{21},c_{31},c_{111} \mid c_{21} c_{31} c_{111}=1, c_{21}^{2}=c_{31}^{3}=c_{111}^{11}=1 \rangle$. Consider a mapping $$\overline{\phi}:\Gamma(h_{\geq 2};-) ~ \rightarrow ~ PSL_2(\mathbb{F}_{11})$$ defined as $\overline{\phi}(\alpha_1)=e_{21} ~ \overline{\phi}(\beta_1)=1$ $$\overline{\phi}(\alpha_2)=e_{31} ~ \overline{\phi}(\beta_2)=1$$ $\overline{\phi}(\alpha_i)= 1 ~ \overline{\phi}(\beta_i)=1$. Then $\overline{\phi}:\Gamma(h_{\geq 2};-) ~ \rightarrow ~ PSL_2(\mathbb{F}_{11})$ is a surface kernel epimorphism. So $(h_{\geq 2};-)$ is a signature of $PSL_2(\mathbb{F}_{11})$.
Now using $Lemma ~4.1$ and $Lemma ~4.2$ we can prove the other remaing cases easily.
\end{proof}

\begin{proof}[Proof of Theorem \ref{main theorem 2}] The proof follows from \ref{1}, \ref{2}, \ref{3}.
\end{proof}

\begin{corollary}
There are no irreducible character $PSL_2(\mathbb{F}_{11})$ that come from the Riemann surfaces.
\end{corollary}

\begin{corollary}
The stable upper genus of $PSL_2(\mathbb{F}_{11})$ is $3508$.
\end{corollary}

\noindent To prove our claim that the stable upper genus of $PSL_2(\mathbb{F}_{11})$ is $3508$, we first have to show that $PSL_2(\mathbb{F}_{11})$ can not act on a compact, connected, orientable surface of genus $3507$ preserving the orientation. If possible let $$ 3507 ~ = ~ 1+660(h-1) + 165a_2 + 220a_3 + 264a_5 + 275a_6+300a_{11}.$$ Then the value of $h$ could be at most $8$. Similarly the values of $a_i$ could be at most $27,~ 20,~ 17,~ 17,~15$ for $i= ~ 2,~ 3,~ 5,~ 6,~11$ respectively. So We will consider  $$0 ~ \leq ~ h ~ \leq ~8$$ $$0 ~ \leq ~ a_2 ~ \leq ~ 27$$ $$0 ~ \leq ~ a_3 ~ \leq ~ 20$$ $$0 ~ \leq ~ a_5 ~ \leq ~ 17$$ $$0 ~ \leq ~ a_6 ~ \leq ~ 17$$ $$0 ~ \leq ~ a_{11} ~ \leq ~ 15.$$

\noindent We execute the following python code to conclude that $PSL_2(\mathbb{F}_{11})$ can not act on a compact, connected, orientable surface of genus $3507$ preserving the orientation.

\lstset{language=Python}
\lstset{frame=lines}
\lstset{caption={$3507$ is not an admissable signature of $PSL_2(\mathbb{F}_{11})$}}
\lstset{label={2nd:code_direct}}
\lstset{basicstyle=\footnotesize}
\begin{lstlisting}
def func2(h,a2,a3,a5,a6,a11):
   return 1+660*(h-1) + 165*a2 + 220*a3 + 264*a5 + 275*a6 + 300*a11



for h in range(8):
   for a2 in range(27):
     for a3 in range(20):
       for a5 in range(17):
         for a6 in range(17):
           for a11 in range(15):
            sol = func2(h,a2,a3,a5,a6,a11)
            if sol >3500:
              if sol < 3550:
                if sol == 3507:
                  print(sol)
                  print(h, a2, a3, a5, a6, a11)
                  
                 


\end{lstlisting}

\noindent To complete the proof of our claim, we have to find out all signatures corresponding to the genus values $3508-3673$ of $PSL_2(\mathbb{F}_{11})$. We execute the following python code to compute all the signature values of $PSL_2(\mathbb{F}_{11})$ corresponding to the genus values $3508-3673$. 

\lstset{language=Python}
\lstset{frame=lines}
\lstset{caption={Signatures of $PSL_2(\mathbb{F}_{11})$} corresponding to the genus value $3508-3673$}
\lstset{label={3rd:code_direct}}
\lstset{basicstyle=\footnotesize}
\begin{lstlisting}
def func2(h,a2,a3,a5,a6,a11):
   return 1+660*(h-1) + 165*a2 + 220*a3 + 264*a5 + 275*a6 + 300*a11


sol_arr = []
const_arr = []
for h in range(10):
   for a2 in range(15):
     for a3 in range(15):
       for a5 in range(15):
         for a6 in range(15):
           for a11 in range(15):
            sol = func2(h,a2,a3,a5,a6,a11)
            if sol >3507:
              if sol < 3674:
               #print(sol)
               sol_arr += [sol]
               const_arr += [[h,a2,a3,a5,a6,a11]]



color_dictionary = dict(zip(sol_arr, const_arr))

sort_orders = sorted(color_dictionary.items(), key=lambda x: x[0])

for i in sort_orders:
	print(i[0], i[1])

\end{lstlisting} 

\noindent Now we have to prove that $PSL_2(\mathbb{F}_{11})$ can act on all compact, connected, orientable surface of genus $g ~ \geq ~ 3674$ preserving the orientation. Let $g ~ \geq 3674$, and $\Sigma_{g}$ be a compact, connected, orientable surface of genus $g$. So we have $$  g-3508 ~ \equiv ~ s ~ (mod ~165) ~ \text{ where } ~1 ~ \leq ~ s ~ \leq 164.$$ Then $g ~ = ~ l+n.165$ where $ l ~= 3508+ s$. We know the signature corresponding to the genus $l$ as $3508~\leq l~ \leq 3673$ and let it be $(h;m_2,~m_3,~m_5,~m_6,m_{11})$. Then the signature corresponding to the genus $g$ is $(h;m_2+n,~m_3,~m_5,~m_6,m_{11})$. In this way we can find signature corresponding to genus $g ~ \geq 3674$. This completes the proof of our claim.

\begin{corollary}
For a signature $(h;2^{[a_2]},3^{[a_3]},5^{[a_5]},6^{[a_6]},11^{[a_{11}]})$ we have a branched covering $\Sigma_g$ of $\Sigma_h$ with $Deck(\Sigma_g/\Sigma_h)=PSL_2(\mathbb{F}_{11})$ where $g=1+660(h-1)+165.a_2+220.a_3+264.a_5+275.a_6+300.a_{11
}$ and $g \geq ~ 26$. So for every surface $\Sigma_g$, we will have a unique irreducible polynomial $F(w)$ over $\mathcal{M}(\Sigma_h)$ using $Theorem 2.6$ such that $\mathcal{M}(\Sigma_g) ~ \cong ~ \mathcal{M}(\Sigma_h)(F(w))$.  
\end{corollary}

\section{Graded Monoid}
\noindent Let $S_i ~ := ~ \lbrace (i;m_1,m_2,\dots,m_r)|i \geq 0 \rbrace$ be the set of signatures of $PSL_2(\mathbb{F}_p)$ when the orbifold genus is $i$. Here $p$ is either 7 or 11. Now we can define $\odot$ on $S_i$ as $$(i;m_1,m_2,\dots,m_r)~ \odot ~(i;n_1,n_2,\dots,n_s) ~ := ~ (i;m_1,m_2,\dots,m_r,n_1,n_2,\dots,n_s) ~ := ~ (i;k_1,k_2,\dots,k_{r+s})$$ where $k_i \geq k_j$ for $i \geq j$ and $k_i ~ \in ~ \lbrace m_i,n_j|i=1,2,\dots,r ~\text{and} ~ j=1,2,\dots, s\rbrace $. Now $(S_i,\odot)$ is a monoid where the free action on the surface $\Sigma_i$ is the identity element in $S_i$.\\
Now consider two signatures $(i;m_1,m_2,\dots,m_r)$ and $(j;n_1,n_2,\dots,n_s)$ of the group $PSL_2(\mathbb{F}_7)$ corresponding to the action on $\Sigma_g$ and $\Sigma_h$ respectively. Now find two disks $D_{g}$ and $D_{h}$ such that $aD_{g}$ and $aD_{h}$ are disjoint for all $a \in PSL_2(\mathbb{F}_7)$. Then take all such disks $aD_{g}$ and $aD_{h}$ from $\Sigma_g$ and $\Sigma_h$ respectively and perform connected sum operation by identifying $\partial(aD_{g})$ and $\partial(aD_h)$ for all $a \in ~ PSL_2(\mathbb{F}_7)$. So we have $\Sigma_g$ and $\Sigma_h$ are joined by $168$ tubes, and as a result we have the surface $\Sigma_{g+h+167}$. We can extend the action of the group $PSL_2(\mathbb{F}_7)$ on $\Sigma_{g+h+167}$ by restricting the action of $PSL_2(\mathbb{F}_7)$ on $\Sigma_g$, $\Sigma_h$ and permuting the tubes. So we get $(i+j;m_1,m_2,\dots,m_r,n_1,n_2,\dots,n_s)$ the resultant signature.\\
Now we define $$ \oplus : ~ S_i ~ \otimes ~ S_j ~ \rightarrow ~ S_{i+j} $$ as $(i;m_1,m_2,\dots,m_r)\oplus (j;n_1,n_2,\dots,n_s)= (i+j;m_1,m_2,\dots,m_r,n_1,n_2,\dots,n_s)=(i+j;k_1,k_2,\dots,k_{r+s})$ where $k_i \geq k_j$ for $i >j$. This shows that $\mathcal{S} ~ := ~ (S_{i_{\geq 0}})$ is a graded monoid. 
\section{Acknowledgment}
This research was supported/partially supported by the Council of Scientific \& Industrial Research. We are thankful to our colleagues and seniors Dr. Debangshu Mukherjee, Dr. Suratno Basu, Dr. Biswajit Ransingh, Dr. Manish Kumar Pandey, Dr. Kuldeep Saha, Dr. Nupur Patanker, Apeksha Sanghi, Prof. V. Kannan, Mr. Aditya tiwari, Suprokash Hazra, Mr. Raja. M. and Dr. Neeraj Dhanwani who provided expertise that greatly assisted the research.
We are also grateful to Mr. Kaustav Mukherjee for assistance with python techniques, and  Dr. Suratno Basu, Dr. Biswajit Ransingh, Dr. Manish Kumar Pandey who moderated this paper and in that line improved the manuscript significantly.
We have to express our appreciation to Dr. Suratno Basu, and Dr. Biswajit Ransingh for sharing their pearls of wisdom with us during this research.


\begin{thebibliography}{30}
\bibitem{program} The programming part as been exclusively done in Python 3.8 on Apple M1 processor, with total run-time of approximately $8$ hours for the range given in Table \ref{Table:genus_values}.
\end{thebibliography}
\end{document}